\documentclass[11pt]{article}
\usepackage[top=1in, bottom=1in, left=1in, right=1in]{geometry}
\usepackage{tikz}

\usepackage{amsmath,amsfonts,amssymb,amscd,amsthm,mathrsfs}
\usepackage{extarrows,textcomp}
\usepackage{graphicx,subfigure,epstopdf}
\usepackage{diagbox,multirow}
\usepackage{algorithm,algorithmic}
\usepackage{caption}
\usepackage{bm}
\usepackage{float}
\usepackage{mlmath}

\usepackage{xcolor}
\usepackage{makeidx,hyperref}

\newtheorem{thm}{Theorem}[section]

\newtheorem{defi}{Definition}[section]
\newtheorem{lem}{Lemma}[section]
\newtheorem{prop}{Proposition}[section]
\newtheorem{rmk}{Remark}[section]
\newtheorem{assump}{Assumption}[section]

\newtheorem*{thm*}{Theorem}
\numberwithin{equation}{section}

\DeclareMathOperator*{\argmin}{arg\,min}

\newcommand{\width}{m}
\newcommand{\R}{\ifmmode\mathbb{R}\else$\mathbb{R}$\fi}
\newcommand{\N}{\ifmmode\mathbb{N}\else$\mathbb{N}$\fi}
\newcommand{\Z}{\ifmmode\mathbb{Z}\else$\mathbb{Z}$\fi}
\newcommand{\Q}{\ifmmode\mathbb{Q}\else$\mathbb{Q}$\fi}

\newcommand{\xx}{{\bm{x}}}

\usepackage{xspace}

\renewcommand{\epsilon}{\varepsilon}

\makeatletter
\newcommand*{\extendadd}{
  \mathbin{
    \mathpalette\extend@add{}
  }
}
\newcommand*{\extend@add}[2]{
  \ooalign{
    $\m@th#1\leftrightarrow$%
    \vphantom{$\m@th#1\updownarrow$}
    \cr
    \hfil$\m@th#1\updownarrow$\hfil
  }
}
\makeatother

\begin{document}

\title{Two-Layer Neural Networks for Partial Differential Equations:\\ Optimization and Generalization Theory}

\author{Tao Luo$^1$ and Haizhao Yang$^2$
	\vspace{0.1in}\\
	$^1$School of Mathematical Sciences, Institute of Natural Sciences, MOE-LSC,\\ and Qing Yuan Research Institute, \\Shanghai Jiao Tong University, Shanghai, 200240, P.R. China\\
    $^2$Department of Mathematics, Purdue University, West Lafayette, IN 47907,  USA
}

\maketitle

\begin{abstract}
 The problem of solving partial differential equations (PDEs) can be formulated into a least-squares minimization problem, where neural networks are used to parametrize PDE solutions. A global minimizer corresponds to a neural network that solves the given PDE. In this paper, we show that the gradient descent method can identify a global minimizer of the least-squares optimization for solving second-order linear PDEs with two-layer neural networks under the assumption of over-parametrization. We also analyze the generalization error of the least-squares optimization for second-order linear PDEs and two-layer neural networks, when the right-hand-side function of the PDE is in a Barron-type space and the least-squares optimization is regularized with a Barron-type norm, without the over-parametrization assumption.
\end{abstract}



\section{Introduction}\label{sec: introduction}

Deep learning, originated in computer science, has revolutionized many fields of science and engineering recently. This revolution also includes broad applications of deep learning in computational and applied mathematics, e.g., many breakthroughs in solving partial differential equations (PDEs)  \cite{doi:10.1002/cnm.1640100303,712178,RUDD2015,Carleo2017,Han8505,E2017,BERG2018,Khoo2018,RAISSI2019686,
SIRIGNANO2018,Huang2019,SelectNet}. The key idea of these approaches is to reformulate the PDE solution into a global minimizer of an expectation minimization problem, where deep neural networks (DNNs) are applied for discretization and the stochastic gradient descent (SGD) is adopted to solve the minimization problem. These methods probably date back to the 1990s (e.g., see \cite{doi:10.1002/cnm.1640100303,712178}) and were revisited recently \cite{RUDD2015,Han8505,E2017,BERG2018,Khoo2018,SIRIGNANO2018,RAISSI2019686} due to the significant development of GPU computing that accelerates DNN computation. Though these approaches have remarkable empirical successes, their theoretical justification remains vastly open. 

For simplicity, let us use a PDE defined on a domain $\Omega$ in a compact form with equality constrains to illustrate the main idea, e.g.,
\begin{equation}\label{eqn:PDE}
    \left\{
    \begin{aligned}
        \fL u & =f\quad \text{in }\Omega,         \\
        \fB u & =g\quad \text{on }\partial\Omega,
    \end{aligned}
    \right.
\end{equation}
where $\fL$ is a differential operator and $\fB$ is the operator for specifying an appropriate boundary condition. In the least squares-type methods, DNNs, denoted as $\phi(\vx;\vtheta)$ with a parameter set $\vtheta$, are applied to parametrize the solution space of the PDE and a best parameter set $\vtheta_\fD$ is identified via minimizing an expectation called the population risk (also known as the population loss):
\begin{equation}\label{eqn:poloss}
    \vtheta_\fD=\argmin_{\vtheta} \RD(\vtheta):=
    \Exp_{\vx\sim U(\Omega)} \left[\ell(\fL\phi(\vx;\vtheta),f(\vx))\right] + \gamma \Exp_{\vx\sim U(\partial\Omega)} \left[\ell(\fB\phi(\vx;\vtheta),g(\vx))\right],
\end{equation}
with a positive parameter $\gamma$ and a loss function typically taken as $\ell(y,y')=\frac{1}{2}|y-y'|^2$, where the expectation are taken with uniform distributions $U(\Omega)$ and $U(\partial\Omega)$ over $\Omega$ and $\partial\Omega$, respectively. To implement the expectation minimization above using the gradient descent method (GD), a discrete set of samples are randomly drawn to obtain an empirical risk (or empirical loss) function
\begin{equation}\label{eqn:emloss}
\RS(\vtheta):=
    \frac{1}{n}\sum_{\{\vx_i\}_{i=1}^n\subset\Omega} \ell(\fL\phi(\vx_i;\vtheta),f(\vx_i)) + \gamma     \frac{1}{n}\sum_{\{\vx_i\}_{i=1}^n\subset \partial\Omega} \ell(\fB\phi(\vx_i;\vtheta),g(\vx_i))
\end{equation}
used in each GD iteration to update $\vtheta$. The set of random samples is usually renewed per iteration resulting in the SGD algorithm for minimizing \eqref{eqn:poloss}. In this paper, we will focus on the case when these samples are fixed in all iterations. There are mainly three theoretical point of view to study the above deep learning-based PDE solver: 
\begin{enumerate}
\item \textbf{Approximation theory:} given a budget of the size of DNNs, e.g. width $\width$ and depth $L$, or a budget of the total number of parameters $N_\mathrm{para}$, what is the accuracy of $\phi(\bm{x};\vtheta_\fD)$ approximating the solution of the PDE? 
\item \textbf{Optimization convergence:} under what condition can gradient descent converges to a global minimizer of \eqref{eqn:poloss} and \eqref{eqn:emloss}?
\item \textbf{Generalization analysis:} if only finitely many samples are available, how good is the global minimizer of \eqref{eqn:emloss} compared to the global minimizer of \eqref{eqn:poloss}?
\end{enumerate}

Deep network approximation theory has shown that DNNs admit powerful approximation capacity. First, DNNs can approximate high-dimensional functions with an appealing approximation rate, e.g., Barron spaces \cite{barron1993,Weinan2019,e2019priori}, Korobov spaces \cite{Hadrien}, band-limited functions \cite{doi:10.1002/mma.5575,bandlimit}, compositional functions \cite{poggio2017,Weinan2019BarronSA}, smooth functions \cite{2019arXiv190609477Y,Shen3,MO}, solution spaces of certain PDEs \cite{HJKN19_814}, and even general continuous functions \cite{Shen4,Shen5}. Second, DNNs can achieve exponential approximation rates, i.e., the approximation error exponentially decays when the number of parameters increases, for target functions in the polynomial spaces \cite{yarotsky2017,bandlimit,Shen3}, the smooth function spaces \cite{bandlimit,DBLP:journals/corr/LiangS16}, the analytic function space \cite{DBLP:journals/corr/abs-1807-00297}, the function space admitting a holomorphic extension to a Bernstein polyellipse \cite{Opschoor2019}, and even general continuous functions \cite{Shen4}. Theories in deep network approximation have provided attractive upper bounds of the accuracy of $\phi(\bm{x};\vtheta_\fD)$ approximating the solution of the PDE in various function spaces. In realistic applications, it might be more interesting to characterize deep network approximation in terms of $\width$ and $L$ simultaneously than the characterization in terms of $N_\mathrm{para}$. We refer reader to \cite{Shen1,Shen2,Shen3,Shen4,Wang} for examples in terms of $\width$ and $L$. 

Though DNNs are powerful in terms of approximation theory, obtaining the best DNN $\phi(\bm{x};\vtheta_\fD)$ in \eqref{eqn:poloss} to approximate the PDE solution is still challenging. It is conjectured that, under certain conditions, SGD is able to identify an approximate global minimizer of \eqref{eqn:poloss} with accuracy depending on $N_\mathrm{para}$ and the sample size $n$. Though deep learning-based PDE solvers have been proposed since the 1990s, there might be no existing literature to investigate this conjecture, to the best of our knowledge. In this paper, assuming that the same set of random samples are used in minimizing \eqref{eqn:emloss}, it is shown that GD can converge to a global minimizer of \eqref{eqn:emloss}, denoted as $\vtheta_S$, for second-order linear PDEs and two-layer neural networks, as long as $N_\mathrm{para}$ is sufficiently large depending on $n$, i.e., in the over-parametrization regime. Furthermore, we will quantify how good the global minimizer $\vtheta_S$ of the empirical loss in \eqref{eqn:emloss} is compared to the global minimizer $\vtheta_\fD$ of the population loss in \eqref{eqn:poloss}, when the empirical loss is regularized with a penalty term using the path norm of $\vtheta$ and the PDE solution is in a Barron-
type space, a variant of the Barron-type space in \cite{barron1993,Weinan2019}. Our analysis is an extension of the seminal work of neural tangent kernels \cite{Arthur18,Simon18,du2018gradient} and the generalization analysis in \cite{barron1993,Weinan2019} for function regression problems to the case of PDE solvers.


Though the convergence of deep learning-based regression under the over-parametrization assumption has been proposed recently \cite{Arthur18,Simon18,MeiE7665,du2018gradient,Yiping20}, we would like to emphasize that the minimization of solving a PDE via \eqref{eqn:poloss} is more difficult and techinical. In the case of solving PDEs, differential operators have changed the optimization objective function considered in the literature. Balancing between the differential operator and the boundary operator makes it more challenging to solve the optimization problem. For example, we consider a second order elliptic equation with variable coefficients, i.e., $\fL u=f$ where $\fL u=\sum_{\alpha,\beta=1}^d A_{\alpha\beta}(\vx)u_{x_{\alpha}x_{\beta}}$. Given a two-layer neural network $\phi(\vx;\vtheta)=\sum_{k=1}^\width a_k\sigma(\vw_k^\T\vx)$ with an activation function $\sigma(z)=\max\{0,\frac{1}{6}z^3\}$ to parametrize the PDE solution, solving the original PDE via deep learning is equivalent to solving a regression problem with another type of neural network  $f(\vx;\vtheta):=\fL\phi(\vx;\vtheta)=\sum_{k=1}^\width a_k\vw_k^\T\mA(\vx)\vw_k\sigma''(\vw_k^\T\vx)$ to fit $f(\vx)$. Note that $\sigma''(z)=\ReLU(z)=\max\{0,z\}$. Thus, the dependence of $f(\vx;\vtheta)$ on $\vw_k$ is essentially cubic rather than linear (more precisely, positive homogeneous). 

The generalization analysis of deep learning-based regression under the over-parametrization assumption was studied recently in \cite{Arthur18,Yuan1,Chen1}. The generalization analysis with a regularization term based on the path norm without the over-parametrization assumption was proposed in \cite{Weinan2019,e2019priori,E_2020}. In the case of PDE solvers, differential operators have enhanced the nonlinearity of the generalization analysis and hence make it more difficult to analyze. In the case of Linear Kolmogorov Equations and parabolic PDEs, examples of generalization analysis of PDE solvers were presented in \cite{DBLP:journals/corr/abs-1809-03062,Han2018ConvergenceOT}. In the case of linear second-order elliptic and parabolic type PDEs, the generalization error of the physics-informed neural network was analyzed in \cite{PINNgen}. However, the generalization analysis for generic PDEs is vastly open. Our attempt is for second-order linear PDEs with variable coefficients. Let us consider the second order elliptic equation with variable coefficients in the above paragraph again. The variable coefficients $A_{\alpha\beta}(\vx)$ lead to highly nonlinearity in the network $f(\vx;\vtheta)$ depending on $\vx$, since we do not make any assumption on the smoothness of $\mA(\vx)$. We develop new analysis of the Rademacher complexity to overcome these difficulties. Unlike existing work, our a priori estimates do not require any truncation on $f(\vx;\vtheta)$ (or $\phi(\vx;\vtheta)$). This is important because a common truncation trick does not lead to the boundedness of $f(\vx;\vtheta)$ in our PDE solver. In fact, if one considers the standard truncation on $\phi(\vx;\vtheta)$, e.g., $\fT_{[0,1]}\phi(\vx;\vtheta):=\min\{\max\{\phi(\vx;\vtheta),0\},1\}$, then $\fL[\fT_{[0,1]}\phi(\vx;\vtheta)]$ might still be unbounded because $\fL$ is a second order differential operator. Another naive trick is to truncate $f(\vx;\vtheta)$, i.e., $\fT_{[0,1]}f(\vx;\vtheta):=\min\{\max\{f(\vx;\vtheta),0\},1\}$. But this does not make sense since we want to find a solution satisfying $\fL \phi(\vx;\vtheta)\approx f(\vx)$ instead of $\fT_{[0,1]}\fL \phi(\vx;\vtheta)\approx f(\vx)$.

This paper will be organized as follows. In Section \ref{sec:DLPDE}, deep learning-based PDE solvers will be introduced in detail. In Section \ref{sec:main}, our main theorems for the convergence and generalization analysis of GD for minimizing \eqref{eqn:emloss} will be presented. In Section \ref{sec:conv_gd}, the proof of the GD convergence theorems will be shown. In Section \ref{sec:gen_err}, the proof of the generalization bound will be given. Finally, we conclude our paper in Section \ref{sec:conc}.

\section{Deep Learning-based PDE Solvers}
\label{sec:DLPDE}

We will introduce deep learning-based PDE solvers with necessary notations in this paper in preparation for our main theorems in Section \ref{sec:main}.

\subsection{Notations, Definitions, and Basic Lemmas}
\label{sec:note}

The main notations of this paper are listed as follows.
\begin{itemize}    
\item Vectors and matrices are denoted in bold font. All vectors are column vectors.
    \item For a parameter set $\Theta$, $\text{vec}\{\Theta\}$ denotes the vector consists of all the elements of $\Theta$.   
    \item $[n]$ denotes $\{1,2,\dots,n\}$. 
    \item $\|\cdot\|_1$ and $\|\cdot\|_\infty$ represent the $\ell_1$ and $\ell_\infty$ norms of a vector, respectively.
    \item Big ``$O$" notation: for any functions $g_1,g_2:\sR\to\sR^+$, $g_1(z)=O(g_2(z))$ as $z\to+\infty$ means that ${g_1(z)}\leq Cg_2(z)$ for some constants $C$, $z_0$ and any $z\geq z_0$.
    \item Small ``$o$" notation: for any functions $g_1,g_2:\sR\to\sR^+$, $g_1(z)=o(g_2(z))$ as $z\to+\infty$ means that $\lim_{z\rightarrow \infty}\frac{f(z)}{g(z)}=0$.
    \item Let $\sigma:\R\to \R$ denote the activation function, e.g., $\sigma(x)=\max\{0,\frac{1}{6}x^3\}$ is the activation function used in this paper.  With the abuse of notations, we define $\sigma:\R^d\to \R^d$ as $\sigma(\xx)=(
          \max\{0,x_1\}, \dots,
          \max\{0,x_d\})^\T$ for any $\xx=(x_1,\dots,x_d)^\T\in \R^d$, where $\T$ denotes the transpose of a matrix. Similarly, for any function $f$ defined on $\R$ and vector $\bm{x}\in\R^d$, $f(\bm{x})=[f(x_1),\dots,f(x_d)]^\T$.
\end{itemize}

Mathematically, DNNs are a form of function parametrization via the compositions of simple non-linear functions \cite{IanYoshuaAaron2016}. Let us focus on the so-called fully connected feed-forward neural network (FNN) defined below. The FNN is a general DNN structure that includes other advanced structures as its special cases, e.g., convolutional neural network \cite{IanYoshuaAaron2016}, ResNet \cite{DBLP:journals/corr/HeZRS15}, and DenseNet \cite{DBLP:journals/corr/HuangLW16a}.

\begin{defi}[Fully connected feed-forward neural network (FNN)]  \label{def:DNN} 
An FNN of depth $L$ defined on $\R^d$ is the composition of $L$ simple nonlinear functions as follows:
\begin{equation*}
    \phi(\vx;\vtheta):=\va^\T \vh^{[L]} \circ \vh^{[L-1]} \circ \cdots \circ \vh^{[1]}(\vx),
\end{equation*}
where $\vh^{[l]}(\vx)=\sigma\left(\mW^{[l]} \vx + \vb^{[l]} \right)$ with $\mW^{[l]} \in \sR^{\width_{l}\times \width_{l-1}}$, $\vb_l \in \sR^{\width_l}$ for $l=1,\dots,L$, $\va\in \sR^{\width_L}$, $\width_0=d$, and $\sigma$ is a non-linear activation function. Each $\vh^{[l]}$ is referred as a hidden layer,  $\width_l$ is the width of the $l$-th layer, and $L$ is called the depth of the FNN. $\vtheta:=\mathrm{vec}\{\va,\{\mW^{[l]},\vb^{[l]}\}_{l=1}^L\}$ denotes the set of all parameters in $\phi$.
\end{defi}
Without loss of generality, we consider FNNs omitting $\vb^{[l]}$'s. In fact, for a network with $\vb^{[l]}$'s, one can simply set $\tilde{\vx}=(\vx^\T,1)^\T$ and $\tilde{\vW}^{[l]}=(\vW^{[l]},\vb^{[l]})$ for each $l\in[L]$, and work on $\vtheta=\mathrm{vec}\{\va,\{\tilde{\mW}^{[l]}\}_{l=1}^L\}$ by noting that $\tilde{\vW}^{[l]}\tilde{\vx}=\mW^{[l]} \vx + \vb^{[l]}$. In this paper, we will focus on networks with $L=1$.

To analyze PDE solvers, we introduce a new kind of Barron functions with their associated Barron norm, and a path norm defined below.

\begin{defi}[Path norm]\label{prop:2L}
    The path norm of a two-layer neural network
    \[
    \phi(\vx;\vtheta)=\sum_{k=1}^\width a_k\sigma(\vw_k^\T\vx),
    \]
    with an activation function $\sigma$ and a parameter set $\vtheta$ is defined as
\[
\|\vtheta\|_{\fP}:=\sum_{j=1}^\width|a_j|\|\bm{w}_j\|_1^3.
\]
\end{defi}


\begin{defi}\label{def:bfun}
    A function $f:\Omega\to\sR$ is called a Barron-type function if $f$ has an integral representation
    \begin{equation*}
        f(\vx)
         = \Exp_{(a,\vw)\sim\rho}a[\vw^\T\mA(\vx)\vw\sigma''(\vw^\T\vx)+\vb^\T(\vx)\vw\sigma'(\vw^\T\vx)+c(\vx)\sigma(\vw^\T\vx)]\quad\text{for all}\quad \vx\in\Omega,
    \end{equation*}
    where $\rho$ is a probability distribution over $\sR^{d+1}$. The associated Barron norm of a Barron-type function is defined as
    \begin{equation*}
        \norm{f}_{\fB}
         := \inf\limits_{\rho\in \fP_f} \left(\Exp_{(a,\vw)\sim\rho}
        \abs{a}^2\norm{\vw}_1^6\right)^{1/2},
    \end{equation*}
    where 
    $
      \fP_f = \{\rho\mid f(\vx) = \Exp_{(a,\vw)\sim\rho}a[\vw^\T\mA(\vx)\vw\sigma''(\vw^\T\vx)+\vb^\T(\vx)\vw\sigma'(\vw^\T\vx)+c(\vx)\sigma(\vw^\T\vx)],\vx\in\Omega\}
    $.
    The Barron-type space is defined as $\fB(\Omega) = \{f:\Omega\to\sR\mid \norm{f}_{\fB} < \infty\}$.
\end{defi}

Since $\RD(\vtheta)$ cannot be realized in realistic applications due to the fact that the empirical loss $\RS(\vtheta)$ of finitely many samples is actually used in the computation, an immediate question is: how well $\phi(\vx;\vtheta_S)\approx \phi(\vx;\vtheta_\fD)$? Here $\vtheta_S$ is a global minimizer when we minimize the empirical loss of $\RS(\vtheta)$. This is the generalization error analysis of deep learning-based PDE solvers and we will use the Rademacher complexity below to estimate the generalization error in terms of $|\RD(\vtheta_S)-\RS(\vtheta_S)|$.

\begin{defi}[The Rademacher complexity of a function class $\fF$] \label{def:rad}
    Given a sample set $S=\{z_1,\dots,z_n\}$ on a domain $\fZ$, and a class $\fF$ of real-valued functions defined on $\fZ$, the empirical Rademacher complexity of $\fF$ on $S$ is defined as
    \[
    \Rad_S(\fF)=\frac{1}{n}\Exp_{\vtau}\left[ \sup_{f\in \fF} \sum_{i=1}^n \tau_i f(z_i)  \right],
    \]
    where $\tau_1$, $\dots$, $\tau_n$ are independent random variables drawn from the Rademacher distribution, i.e., $\Prob(\tau_i=+1)=\Prob(\tau_i=-1)=\frac{1}{2}$ for $i=1,\dots,n$.
\end{defi}

The Rademacher complexity is a basic tool for generalization analysis. In our analysis, we will use several important lemmas and theorems related to it. For the purpose of being self-contained, they are listed as follows.

First, we recall a well-known contraction lemma for the Rademacher complexity. 
\begin{lem}[Contraction lemma \cite{Shalev-Shwartz2014}]\label{lem..RademacherComplexityContraction}
    Suppose that $\psi_i:\sR\to\sR$ is a $C_\mathrm{L}$-Lipschitz function for each $i\in[n]$. 
    For any $\vy\in\sR^n$, let $\vpsi(\vy)=(\psi_1(y_1),\cdots,\psi_n(y_n))^\T$. For an arbitrary set of vector functions $\fF$ of length $n$ on an arbitrary domain $\fZ$ and an arbitrary choice of samples $S=\{\vz_1,\dots,\vz_n\}\subset\mathcal{Z}$, we have
    \begin{align*}
        \Rad_S(\psi\circ \fF)\leq C_\mathrm{L}\Rad_S(\fF).
    \end{align*}
\end{lem}

Second, the Rademacher complexity of linear predictors can be characterized by the lemma below.

\begin{lem}[Rademacher complexity for linear predictors \cite{Shalev-Shwartz2014}]\label{lem..linear}
    Let $\Theta=\{\vw_1,\cdots,\vw_\width\}\in\sR^d$. Let $\fG=\{g(\vw)= \vw^\T\vx:\norm{\vx}_1\leq 1\}$ be the linear function class with parameter $\vx$ whose $\ell^1$ norm is bounded by $1$. Then
    \begin{equation*}
        \Rad_\Theta(\fG)\leq \max_{1\leq k\leq m} \norm{\vw_k}_\infty\sqrt{\frac{2\log(2d)}{\width}}.
    \end{equation*}
\end{lem}

Finally, let us state a general theorem concerning the Rademacher complexity and generalization gap of an arbitrary set of functions $\fF$ on an arbitrary domain $\fZ$, which is essentially given in \cite{Shalev-Shwartz2014}.
\begin{thm}[Rademacher complexity and generalization gap \cite{Shalev-Shwartz2014}]\label{thm..RademacherComplexityGeneralizationGap}
    Suppose that $f$'s in $\fF$ are non-negative and uniformly bounded, i.e., for any $f\in\fF$ and any $\vz\in\fZ$, $0\leq f(\vz)\leq B$. Then for any $\delta\in(0,1)$, with probability at least $1-\delta$ over the choice of $n$ i.i.d. random samples $S=\{\vz_1,\dots,\vz_n\}\subset\mathcal{Z}$, we have
    \begin{align*}
        \sup_{f\in\fF}\Abs{\frac{1}{n}\sum_{i=1}^n f(\vz_i)-\Exp_{\vz}f(\vz)}
        &\leq 2\Exp_{S}\Rad_{S}(\fF)+B\sqrt{\frac{\log(2/\delta)}{2n}},\\
        \sup_{f\in\fF}\Abs{\frac{1}{n}\sum_{i=1}^n f(\vz_i)-\Exp_{\vz}f(\vz)}
        &\leq 2\Rad_{S}(\fF)+3B\sqrt{\frac{\log(4/\delta)}{2n}}.
    \end{align*}
\end{thm}

\subsection{Expectation Minimization}\label{sec:sol}

We will focus on the least-squares method in \eqref{eqn:poloss} for the boundary value problem (BVP) in \eqref{eqn:PDE} to discuss the expectation minimization, though the expectation minimization can either be formulated from the least-squares method \cite{BERG2018,SIRIGNANO2018,RaissiPerdikarisKarniadakis2019} or the variational formulation  \cite{EYu2018,Liao2019}. 
As we shall see in the next subsection, an initial value problem (IVP) can also be formulated into a BVP and solved by the expectation minimization in this subsection. 

The objective function in \eqref{eqn:poloss} consists of two parts: one part for the PDE operator in the domain interior and another part for the boundary condition at the boundary. Therefore, GD has to balance between these two parts and its performance heavily relies on the choice of the parameter $\gamma$ in \eqref{eqn:poloss}. To remove the hyper-parameter $\gamma$ and solve the balancing issue, we will introduce special DNNs in \cite{SelectNet,Deflation} satisfying various boundary conditions by design, i.e., $\mathcal{B}\phi(\vx;\vtheta)=g(\bm{x})$ is always fulfilled on $\partial\Omega$. Then the expectation minimization in \eqref{eqn:poloss} is reduced to 
\begin{equation}\label{eqn:poloss2}
    \vtheta_\fD=\argmin_{\vtheta} \RD(\vtheta):=
    \Exp_{\vx\in\Omega} \left[\ell(\fL\phi(\vx;\vtheta),f(\vx))\right].
\end{equation}
Special neural networks for three types of boundary conditions will be introduced. Without loss of generality, we will take the example of one-dimensional problems on the domain $\Omega=[a, b]$. Networks for more complicated boundary conditions in high-dimensional domains can be constructed similarly. 

\vspace{0.25cm}

  \noindent  \textbf{Case 1. Dirichlet Boundary Conditions:} $u(a)=a_0, \ u(b)=b_0$.

In this case, two special functions $h_1(x)$ and $h_2(x)$ are used to augment a neural network $\tilde{\phi}(x; \vtheta)$ to construct the final neural network $\phi(x; \vtheta)$ as the solution network:
\[
\phi(x; \vtheta) = h_1(x) \tilde{\phi}(x; \vtheta)+h_2(x).
 \]
$h_1(x)$ and $h_2(x)$ are chosen such that $\phi(x; \vtheta)$ automatically satisfies the Dirichlet boundary conditions no matter what $\vtheta$ is. Then $\phi(x; \vtheta)$ is trained to satisfy the differential operator in the interior of the domain $\Omega$ by solving \eqref{eqn:poloss2}.

To achieve this goal, $h_1(x)$ and $h_2(x)$ are constructed for two purposes: 1) construct $h_1(x)$ such that $h_1(x) \tilde{\phi}(x; \vtheta)$ satisfies  the homogeneous Dirichlet boundary condition; 2) construct $h_2(x)$ such that $h_2(x)$ satisfies the given inhomogeneous Dirichlet boundary conditions. Therefore, $h_1(x)$ can be set as
\begin{equation*}
h_1(x)= (x-a)^{p_a}(x-b)^{p_b},
\end{equation*}
where $0<p_a,\  p_b \leq 1$, and $h_2(x)$ can be chosen as 
\begin{equation*}
h_2(x)=(b_0-a_0)(x-a)/(b-a)+a_0.
\end{equation*}
Note that $p_a$ and $p_b$ should be chosen appropriately to avoid introducing a singular function that $\tilde{\phi}(x; \vtheta)$ needs to approximate. For instance, if the exact PDE solution is $u(x)=(x-a)^s(x-b)^sv(x) + h_1(x)$ with $v(x)$ as a smooth function and $s>0$, 
$p_a=p_b>s$ results in $\tilde{\phi}(x; \vtheta)\approx (x-a)^{s-p_a}(x-b)^{s-p_b}v(x)$, which makes the approximation very challenging.

\vspace{0.25cm}
\noindent \textbf{Case 2. Mixed  Boundary Conditions:} $u'(a)=a_0$, $u(b)=b_0$.

Similar to Case $1$, two special functions $h_1(x)$ and $h_2(x)$ are used to augment a neural network $\tilde{\phi}(x; \vtheta)$ to construct the final neural network $\phi(x; \vtheta)$ as the solution network:
\[
\phi(x; \vtheta) = h_1(x) \tilde{\phi}(x; \vtheta)+h_2(x).
 \]
$h_1(x)$ and $h_2(x)$ are chosen such that $\phi(x; \vtheta)$ automatically satisfies the mixed boundary conditions no matter what $\vtheta$ is. Then $\phi(x; \vtheta)$ is trained to satisfy the differential operator in the interior of the domain $\Omega$ by solving \eqref{eqn:poloss2}.

To achieve this goal, $h_1(x)$ and $h_2(x)$ are constructed as
\begin{equation*}
h_1(x)= (x-a)^{p_a}
\end{equation*}
with $1<p_a \leq 2$ and $h_2(x)$ can be chosen as 
\begin{equation*}
h_2(x)=-(b-a)^{p_a}\tilde{\phi}(b;\vtheta)+a_0x+b_0-a_0b.
\end{equation*}

\noindent \textbf{Case 3.  Neumann Boundary Conditions:} $u'(a)=a_0$, $u'(b)=b_0$.

Similar to Case $1$ and $2$, we augment a neural network $\tilde{\phi}(x; \vtheta)$ to construct the final neural network $\phi(x; \vtheta,{c}_1,{c}_2)$ as the solution network:
  \begin{equation*}
  \phi(x;\vtheta, {c}_1, {c}_2)=\exp(\frac{p_a x}{a-b})(x-a)^{p_a}\big((x-b)^{p_b}\tilde{\phi}(x;\vtheta)+ {c}_2\big)+ {c}_1+\frac{(b_0-a_0)}{2(b-a)}(x-a)^2+a_0x.
  \end{equation*}
where $1<p_a,p_b\leq2$,  $ {c}_1$ and $ {c}_2$ are two parameters  to be trained together with $\vtheta$. Then $\phi(x; \vtheta, {c}_1, {c}_2)$ automatically satisfies the Neumann boundary conditions no matter what parameters are and $\phi(x; \vtheta, {c}_1, {c}_2)$ is trained to satisfy the differential operator in the interior of the domain $\Omega$ by solving \eqref{eqn:poloss2}.


\subsection{Scope of Analysis and Applications}\label{sec:scope}
In Section \ref{sec:sol}, we have simplified the optimization problem from \eqref{eqn:poloss} to \eqref{eqn:poloss2} for BVP in \eqref{eqn:PDE}. Now we will show that various initial/boundary value problems can be formulated as a BVP in the form of \eqref{eqn:PDE}. This helps us to simplify the optimization convergence and generalization analysis of deep learning-based PDE solvers to the case of BVP in \eqref{eqn:PDE} solved by \eqref{eqn:poloss2}. The analysis of a larger scope of applications has been naturally included in the analysis of BVPs.

Let us assume that the domain $\Omega\subset\mathbb{R}^d$ is bounded. Typical PDE problems of interest can be summerized as:
\begin{itemize}
  \item Elliptic equation:
  \begin{equation}\label{01_1}
  \begin{split}
  &\fL u(\vx)=f(\vx)\text{~in~}\Omega,\\
  &\mathcal{B} u(\vx)=g_0(\vx)\text{~on~}\partial\Omega.
  \end{split}
  \end{equation}
  \item Parabolic equation:
  \begin{equation}\label{01_2}
  \begin{split}
  &\frac{\partial u(\vx,t)}{\partial t}-\fL u(\vx,t)=f(\vx,t)\text{~in~}\Omega\times(0,T),\\
  &\mathcal{B} u(\vx,t)=g_0(\vx,t)\text{~on~}\partial\Omega\times(0,T),\\
  &u(\vx,0)=h_0(\vx)\text{~in~}\Omega.
  \end{split}
  \end{equation}
  \item Hyperbolic equation:
  \begin{equation}\label{01_3}
  \begin{split}
  &\frac{\partial^2 u(\vx,t)}{\partial t^2}-\fL u(\vx,t)=f(\vx,t)\text{~in~}\Omega\times(0,T),\\
  &\mathcal{B} u(\vx,t)=g_0(\vx,t)\text{~on~}\partial\Omega\times(0,T),\\
  &u(\vx,0)=h_0(\vx),\quad\frac{\partial u(\vx,0)}{\partial t}=h_1(\vx)\text{~in~}\Omega.
  \end{split}
  \end{equation}
\end{itemize}
In the above equations, $u$ is the unknown solution function; $f$, $g_0$, $h_0$, $h_1$ are given data functions; $\fL$ is a spatial differential operator with respect to $x$; $\mathcal{B}$ is a boundary operator specifying a certain type of boundary conditions.

As discussed in \cite{SelectNet}, when the temporal variable $t$ is treated as an extra spatial coordinate, we can unify the above initial/boundary value problems in \eqref{01_1}-\eqref{01_3} in the following form
\begin{equation}\label{01}
\begin{split}
&\fL u(\bm{y})=f(\bm{y})\text{~in~}Q,\\
&\mathcal{B}u(\bm{y})=g(\bm{y})\text{~in~}\Gamma,
\end{split}
\end{equation}
where $\bm{y}$ includes the spatial variable $\vx$ and possibly the temporal variable $t$; $\fL u=f$ represents a generic time-independent PDE; $\mathcal{B}u=g$ specifies the original boundary condition on $\vx$ and possibly the initial condition of $t$; $Q$ and $\Gamma$ are the corresponding new domains of the equations. For the purpose of convenience, we will still use the BVP in \eqref{eqn:PDE} instead of \eqref{01} afterwards. 

Though deep learning-based PDE solvers work for high-order differential equations in general domains, we consider second order differential equations with variable coefficients in $\Omega=[0,1]^d$ in our analysis. The generalization to high-order differential equations and other domains follows straightforwardly and we leave it as future work. We will use the second order differential operator $\fL$ in a non-divergence form
    \begin{equation}\label{eqn:L}
        \fL u = \sum_{\alpha,\beta=1}^d A_{\alpha\beta}(\vx)u_{x_{\alpha}x_{\beta}} + \sum_{\alpha=1}^d {b}_\alpha(\vx)u_{x_\alpha} + c(\vx) u.
    \end{equation}
If $\fL$ is in a divergence form, e.g.,
    \begin{equation*}
        \fL u = \sum_{\alpha,\beta=1}^d \left(A_{\alpha\beta}(\vx)u_{x_{\alpha}}\right)_{x_{\beta}} + \sum_{\alpha=1}^d b_\alpha(\vx)u_{x_{\alpha}} + c(\vx) u,
    \end{equation*}
    then we can represent it in a non-divergence form as
    \begin{equation*}
        \fL u = \sum_{\alpha,\beta=1}^d A_{\alpha\beta}(\vx)u_{x_{\alpha}x_{\beta}} + \sum_{\alpha=1}^d \hat{b}_\alpha(\vx)u_{x_\alpha} + c(\vx) u
    \end{equation*}
    with
    \begin{equation*}
        \hat{b}_\alpha = b_\alpha + \sum_{\beta=1}^d\frac{\partial A_{\alpha\beta}}{\partial x_\beta}.
    \end{equation*}
    
Recall that we introduce two functions $h_1(\vx)$ and $h_2(\vx)$ to augment a neural network $\tilde{\phi}(x; \vtheta)$ to construct the final neural network 
\[
\phi(\vx; \vtheta) = h_1(\vx) \tilde{\phi}(\vx; \vtheta)+h_2(\vx)
 \]
 as the solution network that automatically satisfies given Dirichlet boundary conditions, which makes it sufficient to solve the optimization problem in \eqref{eqn:poloss2} to get the desired neural network. In this case, $\fL \phi(\vx;\vtheta)= f(\vx)$ is equivalent to $\tilde{\fL} \tilde{\phi}(\vx;\vtheta)= \tilde{f}(\vx)$, where
 \[
 \tilde{\fL} =  \sum_{\alpha,\beta=1}^d \tilde{A}_{\alpha\beta}(\vx)  u_{x_{\alpha}x_{\beta}}   + \sum_{\alpha=1}^d \tilde{b}_\alpha(\vx)  u_{x_\alpha} +  \tilde{c}(\vx) ,
 \] 
 \[
 \tilde{A}_{\alpha\beta}(\vx) =A_{\alpha\beta}(\vx)  h_1(\vx),
 \]
 \[
 \tilde{b}_\alpha(\vx) =  {b}_\alpha(\vx) h_1(\vx) +   \sum_{\beta=1}^d \left(A_{\alpha\beta}(\vx) + A_{\beta\alpha}(\vx)\right) \partial_{x_{\beta}} h_1(\vx) ,
 \]
 \[
 \tilde{c}(\vx) =    \sum_{\alpha,\beta=1}^d A_{\alpha\beta}(\vx)  \partial_{x_{\alpha}}\partial_{x_{\beta}} h_1(\vx)  +\sum_{\alpha=1}^d {b}_\alpha(\vx)\partial_{x_\alpha} h_1(\vx)  +  c(\vx) h_1(\vx),
 \]
 and
 \[
 \tilde{f}(\vx)=  f(\vx)-\fL(h_2(\vx)).
 \]
 Therefore, the optimization convergence and generalization analysis of \eqref{eqn:poloss2} is equivalent to 
 \begin{equation}\label{eqn:poloss3}
    \vtheta_\fD=\argmin_{\vtheta} \RD(\vtheta):=
    \Exp_{\vx\in\Omega} \left[\ell(\tilde{\fL}\tilde{\phi}(\vx;\vtheta),\tilde{f}(\vx))\right],
\end{equation}
 which gives
 \[
 \phi(\vx; \vtheta_\fD) = h_1(\vx) \tilde{\phi}(\vx; \vtheta_\fD)+h_2(\vx)
 \]
 as a best solution to the PDE in \eqref{eqn:PDE} parametrized by DNNs. The corresponding empirical risk is
 \begin{equation}\label{eqn:emloss3}
\RS(\vtheta):=
    \frac{1}{n}\sum_{\{\vx_i\}_{i=1}^n\subset\Omega} \ell(\tilde{\fL}\tilde{\phi}(\vx_i;\vtheta),\tilde{f}(\vx_i)),
\end{equation}
which gives $\vtheta_S=\argmin_{\vtheta} \RS(\vtheta)$ and 
 \[
 \phi(\vx; \vtheta_S) = h_1(\vx) \tilde{\phi}(\vx; \vtheta_S)+h_2(\vx).
 \]
 
Similarly, in the case of other two types of boundary conditions, the corresponding optimization problem in \eqref{eqn:poloss} can also be transformed to \eqref{eqn:poloss3} and its discretization in \eqref{eqn:emloss3} with an appropriate differential operator $\tilde{\fL}$ and a right-hand-side function $\tilde{f}$.

In sum, the discussion in Section \ref{sec:sol} and here indicates that the optimization and generalization analysis of deep learning-based PDE solvers for various IVPs and BVPs with different boundary conditions can be reduced to the analysis of \eqref{eqn:poloss3} and \eqref{eqn:emloss3} with $\tilde{\fL}$ in a non-divergence form. In the next section, we will present our main theorems for this analysis. For simplicity, we will still use the notation of $\fL$ and $f$ instead of $\tilde{\fL}$ and $\tilde{f}$ in our analysis afterwards.

\section{Main Results}
\label{sec:main}

In this section, we introduce our main results on the convergence of GD and the generalization error of neural network-based least-squares solvers for PDEs using two-layer neural networks on $\Omega=[0,1]^d$. Throughout our analysis, we we assume $|f|\leq 1$ and focus on second-order differential operators $\fL$ given in \eqref{eqn:L} satisfying the assumption below.

\begin{assump}[Symmetry and boundedness of $\fL$]\label{asp:bound}
Throughout the analysis of this paper, we assume $\fL$ in \eqref{eqn:L} satisfies the condition: there exists $M\geq 1$\footnote{The upper bound $M$ is not necessarily greater than $1$. We set this for simplicity.} such that for all $\vx\in\Omega=[0,1]^d$, $\alpha,\beta\in[d]$, we have $A_{\alpha\beta}=A_{\beta\alpha}$ 
\begin{align}\label{eqn:introM}
    \abs{A_{\alpha\beta}(\vx)} \leq M,\quad \abs{b_\alpha(\vx)}\leq M,\quad \text{and}\quad\abs{c(\vx)}\leq M.
\end{align}
\end{assump}

First, we show that, under suitable assumptions, the emprical risk $\RS(\vtheta)$ of the PDE solution represented by an over-parametrized two-layer neural networks converges to zero, i.e., achieving a global minimizer, with a linear convergence rate by GD. In particular, as discussed in Section \ref{sec:DLPDE}, it is sufficient to prove the convergence for minimizing the empirical loss
 \begin{equation}\label{eqn:emloss4}
    \vtheta_S=\argmin_{\vtheta} \RS(\vtheta):=
    \frac{1}{n}\sum_{S=\{\vx_i\}_{i=1}^n\subset\Omega} \ell({\fL}\phi(\vx_i;\vtheta),{f}(\vx_i)),
\end{equation}
where $S:=\{\vx_i\}_{i=1}^n$ is a given set of i.i.d. samples with the uniform distribution $\fD$ over $\Omega=[0,1]^d$, and the two-layer neural network used here is constructed as
\begin{equation}\label{eqn:NN}
    \phi(\vx;\vtheta)=\sum_{k=1}^\width a_k\sigma(\vw_k^\T\vx),
\end{equation}
where for $k\in[\width]$, $a_k\in\sR$, $\vw_k\in\sR^d$, $\vtheta=\mathrm{vec}\{a_k,\vw_k\}_{k=1}^\width$, and $\sigma(x)=\max\{\frac{1}{6}x^3,0\}$. Our main result of the linear convergence rate is summarized in Theorem \ref{thm:lcr} below.

\begin{thm}[Linear convergence rate] \label{thm:lcr}
Let $\vtheta^0 := \mathrm{vec}{\{a_k^0, \vw_k^0\}}_{k=1}^\width$ at the GD initialization for solving \eqref{eqn:emloss4}, where $a^0_k \sim \mathcal{N}(0, \gamma^2)$ and $\vw_k^0 \sim \mathcal{N}(\vzero, \mI_d)$ with any $\gamma\in(0,1)$. Let $C_d := \Exp\norm{\vw}_1^{12}<+\infty$ with $\vw\sim \mathcal{N}(\vzero,\mI_d)$ and $\lambda_S$ be a positive constant in Assumption \ref{assump..LambdaMin}. For any $\delta\in(0,1)$, if
    \begin{align}
        \width\geq\max\Bigg\{ & \frac{512n^4M^4 C_d}{\lambda_S^2\delta}, \frac{200\sqrt{2}Md^3n\log(4\width(d+1)/\delta)\sqrt{\RS(\vtheta^0)}}{\lambda_S}, \\
                         & \frac{2^{23}M^3d^9n^2(\log(4\width(d+1)/\delta))^{4}\sqrt{\RS(\vtheta^0)}}{\lambda_S^2}\Bigg\},
    \end{align}
    then with probability at least $1-\delta$ over the random initialization $\vtheta^0$, we have, for all $t\geq 0$,
    \begin{equation*}
        \RS(\vtheta(t))\leq\exp\left(-\frac{\width\lambda_S t}{n}\right)\RS(\vtheta^0).
    \end{equation*}
\end{thm}
\begin{rmk}
    For the estimate of $R_S(\vtheta^0)$, see Lemma \ref{lem2}. In particular, if $\gamma=O(\frac{1}{\sqrt{\width}(\log \width)^2})$, then $R_S(\vtheta^0)=O(1)$. One may also use the Anti-Symmetrical Initialization (ASI) \cite{zhang2019}, a general but simple trick that ensures $R_S(\vtheta^0)\leq \frac{1}{2}$.
\end{rmk}
Second, we prove that the a posteriori generalization error $\abs{\RD(\vtheta)-\RS(\vtheta)}$ is bounded by $O\left(\frac{\|\vtheta\|_{\fP}^2\log\norm{\vtheta}_{\fP}}{\sqrt{n}}\right)$, where $\|\vtheta\|_{\fP}$ is the path norm introduced in Definition \ref{prop:2L}, and the a priori generalization error $\RD(\vtheta_{S,\lambda})$ is bounded by $O\left( \frac{\norm{f}_{\fB}^2}{\width} \right)+ O\left( \frac{\norm{f}_{\fB}^2\log\norm{f}_{\fB}}{\sqrt{n}} \right)$, where $\norm{f}_{\fB}$ is the Barron norm for Barron-type functions $f(\vx)$ introduced in Definition \ref{def:bfun}, and $\vtheta_{S,\lambda}$ is a global minimizer of a regularized empirical loss using the path norm. Our results of the generalization errors can be summarized in Theorems \ref{thm:gen1} and \ref{thm:gen2} below.

\begin{thm}[A posteriori generalization bound]\label{thm:gen1}
    For any $\delta\in(0,1)$, with probability at least $1-\delta$ over the choice of random samples $S:=\{\vx_i\}_{i=1}^n$ in \eqref{eqn:emloss4}, for any two-layer neural network $\phi(\vx;\vtheta)$ in \eqref{eqn:NN}, we have
    \begin{equation*}
        \abs{\RD(\vtheta)-\RS(\vtheta)}
        \leq \frac{(\norm{\vtheta}_{\fP}+1)^2}{\sqrt{n}}2M^2(14d^2\sqrt{2\log(2d)}+\log[\pi(\norm{\vtheta}_{\fP}+1)]+\sqrt{2\log(1/3\delta)}).
    \end{equation*}
\end{thm}

\begin{thm}[A priori generalization bound]\label{thm:gen2}
    Suppose that $f(\vx)$ is in the Barron-type space $\fB([0,1]^d)$ and $\lambda\geq 4M^2[2+14d^2\sqrt{2\log(2d)}+\sqrt{2\log(2/3\delta)}]$. Let
    \begin{equation*}
        \vtheta_{S,\lambda}=\arg\min_{\vtheta}J_{S,\lambda}(\vtheta):=\RS(\vtheta)+\frac{\lambda}{\sqrt{n}}\norm{\vtheta}_{\fP}^2\log[\pi(\norm{\vtheta}_{\fP}+1)].
    \end{equation*}
    Then for any $\delta\in(0,1)$, with probability at least $1-\delta$ over the choice of random samples $S:=\{\vx_i\}_{i=1}^n$ in \eqref{eqn:emloss4}, we have
    \begin{align}
        \RD(\vtheta_{S,\lambda})
         & :=\Exp_{\vx\sim\fD}\tfrac{1}{2}(\fL \phi(\vx;\vtheta_{S,\lambda})-f(\vx))^2 \nonumber\\
         & \leq \frac{6M^2\norm{f}_{\fB}^2}{\width}+
        \frac{\norm{f}_{\fB}^2+1}{\sqrt{n}}(4\lambda+16M^2)\left\{\log[\pi(2\norm{f}_{\fB}+1)]+14d^2\sqrt{\log(2d)}+\sqrt{\log(2/3\delta)}\right\}.
    \end{align}
\end{thm}

The proof of Theorem \ref{thm:lcr} will be given in Section~\ref{sec:conv_gd} and the proofs of Theorems \ref{thm:gen1} and \ref{thm:gen2} will be presented in Section~\ref{sec:gen_err}.

\section{Global Convergence of Gradient Descent}\label{sec:conv_gd}

In this section, we will prove the global convergence of GD with a linear convergence rate for deep learning-based PDE solvers as stated in Theorem \ref{thm:lcr}. We will first summarize the notations and assumptions for the proof of Theorem \ref{thm:lcr} in Section \ref{sec:notelcr}. Several important lemmas will be proved in Section \ref{sec:lemlcr}. Finally, Theorem \ref{thm:lcr} is proved in Section \ref{sec:thmlcr}.

\subsection{Notations and Main Ideas}
\label{sec:notelcr}


Let us first summarize the notations and assumptions used in the proof of Theorem \ref{thm:lcr}. 

Recall that we use the two-layer neural network $\phi(\vx;\vtheta)$ in \eqref{eqn:NN} with $\vtheta=\mathrm{vec}\{a_k,\vw_k\}_{k=1}^\width$. In the GD iteration, we use $t$ to denote the iteration or the artificial time variable in the gradient flow. Hence, we define the following notations for the evolution of parameters at time $t$:
\begin{equation*}
     a^t_k 
     := a_k(t), \quad 
     \vw_k^t 
     := \vw_k(t),\quad
     \vtheta^t 
     := \vtheta(t) := \mathrm{vec}{\{a_k^t, \vw_k^t\}}_{k=1}^\width.
\end{equation*}
In the analysis, we also use $ \bar{a}^t := \bar{a}(t):=\gamma^{-1}a_k(t)$ with $0<\gamma<1$, e.g., $\gamma=\frac{1}{\sqrt{\width}}$ or $\gamma=\frac{1}{\width}$. $\bar{\vtheta}(t)$ means $\mathrm{vec}{\{\bar{a}_k^t, \vw_k^t\}}_{k=1}^\width$. Similarly, we can introduce $t$ to other functions or variables depending on $\vtheta(t)$. When the dependency of $t$ is clear, we will drop the index $t$. In the initialization of GD, we set 
\begin{equation}\label{eqn:init}
     a^0_k 
     := a_k(0)\sim \mathcal{N}(0, \gamma^2), \quad 
     \vw_k^0 
     := \vw_k(0)\sim \mathcal{N}(\vzero, \mI_d),\quad
     \vtheta^0 
     := \vtheta(0) := \mathrm{vec}{\{a_k^0, \vw_k^0\}}_{k=1}^\width.
\end{equation}

Note that we use $\sigma(x)=\max\{\frac{1}{6}x^3,0\}$ as the activation of our two-layer neural network. Therefore, $\sigma'(x)=\max\{\frac{1}{2}x^2,0\}$, and $\sigma''(x)=\ReLU(x)=\max\{x,0\}$. For simplicity, we define
\begin{eqnarray}\label{eqn:NN2}
    f_{\vtheta}(\vx)
     & := &f(\vx;\vtheta):=\fL \phi(\vx;\vtheta) \nonumber\\
     & =& \sum_{k=1}^\width a_k[\vw_k^\T\mA(\vx)\vw_k\sigma''(\vw_k^\T\vx)+\vb^\T(\vx)\vw_k\sigma'(\vw_k^\T\vx)+c(\vx)\sigma(\vw_k^\T\vx)],
\end{eqnarray}
which can be treated as a special two-layer neural network for a regression problem $f_{\vtheta}(\vx)\approx f(\vx)$.

For simplicity, we denote $e_i = f_{\vtheta}(\vx_i) - f(\vx_i)$ for $i\in[n]$ and $\ve = (e_1, e_2, \ldots, e_n)^{\T}$. Then the empirical risk can be written as
\begin{equation*}
    \RS(\vtheta) = \frac{1}{2n}\sum_{i=1}^n{\left(f_{\vtheta}(\vx_i) - f(\vx_i)\right)}^2 = \frac{1}{2n}\ve^{\T}\ve.
\end{equation*}
Hence, the GD dynamics is
\begin{equation}\label{eqn:gd_dyn}
    \dot{\vtheta} = -\nabla_{\vtheta}\RS(\vtheta),
\end{equation}
or equivalently in terms of $a_k$ and $\vw_k$ as follows:
\begin{align}
    \dot{a}_k = -\nabla_{a_k}\RS(\vtheta) 
    &= -\frac{1}{n}\sum_{i=1}^n e_i \left[\vw_k^{\T}\mA(\vx_i)\vw_k\sigma''(\vw_k^\T\vx_i) + \vb^{\T}(\vx_i)\vw_k\sigma'(\vw_k^\T\vx_i) + c(\vx_i)\sigma(\vw_k^\T\vx_i)\right],\nonumber\\
    \dot{\vw}_k = -\nabla_{\vw_k}\RS(\vtheta) 
    &= -\frac{1}{n}\sum_{i=1}^n e_i a_k\Big[2\mA(\vx_i)\vw_k\sigma''(\vw_k^\T\vx_i) + \vw_k^{\T}\mA(\vx_i)\vw_k\sigma^{(3)}(\vw_k^\T\vx_i)\vx_i\nonumber\\
    &~~~~+ \sigma'(\vw_k^\T\vx_i)\vb(\vx_i) + \vb^{\T}(\vx_i)\vw_k\sigma''(\vw_k^\T\vx_i)\vx_i + c(\vw_i)\sigma'(\vw_k^\T\vx_i)\vx_i\Big].\nonumber
\end{align}

Adopting the neuron tangent kernel point of view \cite{Arthur18}, in the case of a two-layer neural network with an infinite width, the corresponding kernels $k^{(a)}$ for parameters in the last linear transform and $k^{(w)}$ for parameters in the first layer are functions from $\Omega\times\Omega$ to $\sR$ defined by
\begin{align}
    k^{(a)}(\vx,\vx') 
    &:= \Exp_{\vw\sim \mathcal{N}(\vzero,\mI_d)}g^{(a)}(\vw;\vx,\vx'),\nonumber\\
    k^{(w)}(\vx,\vx')
    &:= \Exp_{(a,\vw)\sim \mathcal{N}(\vzero,\mI_{d+1})}g^{(w)}(a,\vw;\vx,\vx'),\nonumber
\end{align}
where
\begin{align}
    g^{(a)}(\vw;\vx,\vx')
    &:= \left[\vw^{\T}\mA(\vx)\vw\sigma''(\vw^\T\vx) + \vb^{\T}(\vx)\vw\sigma'(\vw^\T\vx) + c(\vx)\sigma(\vw^\T\vx)\right]\nonumber\\
    &~~~~\cdot\left[\vw^{\T}\mA(\vx')\vw\sigma''(\vw^\T\vx') + \vb^{\T}(\vx')\vw\sigma'(\vw^\T\vx') + c(\vx')\sigma(\vw^\T\vx')\right],\nonumber\\
    g^{(w)}(a,\vw;\vx,\vx')
    &:= a^2\big[2\mA(\vx)\vw\sigma''(\vw^\T\vx)+\vw^\T\mA(\vx)\vw\sigma^{(3)}(\vw^\T\vx)\vx+\sigma'(\vw^\T\vx)\vb(\vx)\nonumber\\&~~~~+\vb^\T(\vx)\vw\sigma''(\vw^\T\vx)\vx +c(\vw)\sigma'(\vw^\T\vx)\vx\big]\cdot
    \big[2\mA(\vx')\vw\sigma''(\vw^\T\vx')\nonumber\\
    &~~~~+\vw^\T\mA(\vx')\vw\sigma^{(3)}(\vw^\T\vx')\vx'+\sigma'(\vw^\T\vx')\vb(\vx')\nonumber\nonumber\\
    &~~~~+\vb^\T(\vx')\vw\sigma''(\vw^\T\vx')\vx'+c(\vw)\sigma'(\vw^\T\vx')\vx'\big].\nonumber
\end{align}
These kernels evaluated at $n\times n$ pairs of samples lead to $n\times n$ Gram matrices $\mK^{(a)}$ and $\mK^{(\vw)}$ with $K^{(a)}_{ij} = k^{(a)}(\vx_i,\vx_j)$ and $K^{(w)}_{ij} = k^{(w)}(\vx_i,\vx_j)$, respectively. Our analysis requires the matrix $\mK^{(a)}$ to be positive definite, which has been verified for regression problems under mild conditions on random training data $S=\{\vx_i\}_{i=1}^n$ and can be generalized to our case. Hence, we assume this as follows for simplicity.
\begin{assump}\label{assump..LambdaMin}
    We assume that
    \begin{equation*}
        \lambda_S := \lambda_{\min}\left(\mK^{(a)}\right) > 0.
    \end{equation*}
\end{assump}
For a two-layer neural network with $\width$ neurons, the $n\times n$ Gram matrix $\mG(\vtheta)=\mG^{(a)}(\vtheta)+\mG^{(w)}(\vtheta)$ is given by the following expressions for the $(i,j)$-th entry
\begin{align*}
    \mG_{ij}^{(a)}(\vtheta) 
    &:= \frac{1}{\width}\sum_{k=1}^\width g^{(a)}(\vw_k;\vx_i,\vx_j),\\
    \mG_{ij}^{(w)}(\vtheta)
    &:= \frac{1}{\width}\sum_{k=1}^\width g^{(w)}(a_k,\vw_k;\vx_i,\vx_j).\nonumber
\end{align*}
Clearly, $\mG^{(a)}(\vtheta)$ and $\mG^{(w)}(\vtheta)$ are both positive semi-definite for any $\vtheta$. By using the Gram matrix $\mG(\vtheta)$, we have the following evolution equations to understand the dynamics of GD:
\begin{equation*}
    \frac{\D }{\D t}f_{\vtheta}(\vx_i)
    =-\frac{1}{n}\sum_{j=1}^n\mG_{ij}(\vtheta)(f_{\vtheta}(\vx_j)-f(\vx_j))
\end{equation*}
and
\begin{equation}\label{eqn:boundS}
        \frac{\D}{\D t}\RS(\vtheta) = - \norm{\nabla_{\vtheta}\RS(\vtheta)}^2_2 =-\frac{\width}{n^2}\ve^{\T}\mG(\vtheta) \ve \leq - \frac{\width}{n^2}\ve^{\T}\mG^{(a)}(\vtheta)\ve.
\end{equation}

Our goal is to show that the above evolution equation has a solution $f_{\vtheta}(\vx_i)$ converging to $f(\vx_i)$ for all training samples $\vx_i$, or equivalently, to show that $\RS(\vtheta)$ converges to zero. These goals are true if the smallest eigenvalue $    \lambda_{\min}\left(\mG^{(a)}(\vtheta)\right)$ of $\mG^{(a)}(\vtheta)$ has a positive lower bound uniformly in $t$, since in this case we can solve \eqref{eqn:boundS} and bound $\RS(\vtheta)$ with a function in $t$ converging to zero when $t\rightarrow\infty$ as shown in Lemma \ref{lem:exp_RS}. In fact, a uniform lower bound of $    \lambda_{\min}\left(\mG^{(a)}(\vtheta)\right)$ can be $\frac{1}{2}\lambda_S$, which can be proved in the following three steps:

\begin{itemize}
    \item (\textbf{Initial phase}) By Assumption \ref{assump..LambdaMin} of $\mK^{(a)}$, we can show $    \lambda_{\min}\left(\mG^{(a)}(\vtheta(0))\right)\approx \lambda_S$ in Lemma \ref{lem:lambda_min} using the observation that $\mK^{(a)}_{ij}$ is the mean of $g(\vw;\vx_i,\vx_j)$ over the normal random variable $\vw$, while $\mG^{(a)}_{ij}(\vtheta(0))$ is the mean of $g(\vw;\vx_i,\vx_j)$ with $\width$ independent realizations. 
    \item (\textbf{Evolution phase}) The GD dynamics results in $\vtheta(t)\approx \vtheta(0)$ under the assumption of over-parametrization as shown in Lemma \ref{prop:a_w}, which indicates that
    \[
    \lambda_{\min}\left(\mG^{(a)}(\vtheta(0))\right)\approx \lambda_{\min}\left(\mG^{(a)}(\vtheta(t))\right).
    \]
    \item (\textbf{Final phase}) To show the uniform bound $\lambda_{\min}\left(\mG^{(a)}(\vtheta(t))\right)\geq \frac{1}{2}\lambda_S$ for all $t\geq 0$, we introduce a stopping time $t^*$ via
\begin{equation}\label{eqn:ts}
    t^* = \inf\{t \mid \vtheta(t)\notin \mathcal{M}(\vtheta^0)\},
\end{equation}
where
\begin{equation}\label{eqn:Mtheta}
    \mathcal{M}(\vtheta^0) := \left\{\vtheta \mid \norm{\mG^{(a)}(\vtheta) - \mG^{(a)}(\vtheta^0)}_{\mathrm{F}}\leq \frac{1}{4}\lambda_S\right\},
\end{equation}
and show that $t^*$ is in fact equal to infinity in the final proof of Theorem \ref{thm:lcr} in Section \ref{sec:thmlcr}. 
\end{itemize}

\subsection{Proofs of Lemmas for Theorem \ref{thm:lcr}}
\label{sec:lemlcr}

In this subsection, we will prove several lemmas in preparation for the proof of Theorem \ref{thm:lcr}.

\begin{lem}\label{lem1}
    For any $\delta\in(0,1)$ with probability at least $1-\delta$ over the random initialization in \eqref{eqn:init}, we have
    \begin{equation}\label{eqn:lem1}
        \begin{aligned}
            \max\limits_{k\in[\width]}\left\{\abs{\bar{a}_k^0},\; \norm{\vw^0_k}_\infty\right\}
             & \leq\sqrt{2
            \log\frac{2\width(d+1)}{\delta}},      \\
            \max\limits_{k\in[\width]}\left\{\abs{a_k^0}\right\}
             & \leq \gamma\sqrt{2
                \log\frac{2\width(d+1)}{\delta}}.
        \end{aligned}
    \end{equation}
\end{lem}
\begin{proof}
    If $\rX \sim \mathcal{N}(0, 1)$, then $\Prob(\abs{\rX} > \eps) \leq 2\E^{-\frac{1}{2}\eps^2}$ for all $\eps > 0$. Since $\bar{a}^0_k\sim \mathcal{N}(0,1)$, ${(\vw_k^0)}_{\alpha}\sim \mathcal{N}(0,1)$ for $k\in[\width], \alpha \in[d]$, and they are all independent, by setting
    \begin{equation*}
        \eps = \sqrt{2\log\frac{2\width(d+1)}{\delta}},
    \end{equation*}
    one can obtain
    \begin{equation*}
        \begin{aligned}
            \Prob\left(\max\limits_{k\in[\width]}\left\{\abs{\bar{a}_k^0},\norm{\vw^0_k}_\infty\right \}>\eps\right)
             & = \Prob\left(\left(\bigcup\limits_{k\in[\width]}\left\{\abs{\bar{a}_k^0}>\eps\right\}\right)\bigcup\left(\bigcup\limits_{k\in[\width],\alpha\in[d]}\left\{\abs{{(\vw_k^0)}_{\alpha}}>\eps\right\}\right)\right) \\
             & \leq \sum_{k=1}^\width \Prob\left(\abs{\bar{a}_k^0}>\eps\right) + \sum_{k=1}^\width\sum_{\alpha=1}^d \Prob\left(\abs{{(\vw^0_k)}_{\alpha}}>\eps\right)                             \\
             & \leq 2\width \E^{-\frac{1}{2}\eps^2} + 2\width d \E^{-\frac{1}{2}\eps^2}                                                                                                       \\
             & = 2\width(d+1)\E^{-\frac{1}{2}\eps^2}                                                                                                                                    \\
             & = \delta,
        \end{aligned}
    \end{equation*}
    which implies the conclusions of this lemma.
\end{proof}

\begin{lem}\label{lem2}
    For any $\delta\in(0,1)$ with probability at least $1-\delta$ over the random initialization in \eqref{eqn:init}, we have
    \begin{equation*}
        \RS(\vtheta^0)\leq\frac{1}{2}\left(1 + 32\gamma\sqrt{\width}Md^3\left(\log\frac{4\width(d+1)}{\delta}\right)^2\left(\sqrt{2\log(2d)}+\sqrt{2\log(8/\delta)}\right)\right)^2,
    \end{equation*}
\end{lem}
\begin{proof}
    From Lemma~\ref{lem1} we know that with probability at least $1-\delta/2$,
    \begin{equation*}
        \abs{\bar{a}_k^0} \leq \sqrt{2\log\frac{4\width(d+1)}{\delta}} \quad \text{and} \quad \norm{\vw_k^0}_1 \leq d\sqrt{2\log\frac{4\width(d+1)}{\delta}}.
    \end{equation*}
        Let
    \begin{equation*}
        \fH = \{h(\bar{a},\vw;\vx) \mid h(\bar{a},\vw;\vx) = \bar{a}\left[\vw^{\T}\mA(\vx)\vw \sigma''(\vw^\T\vx) + \vb^{\T}(\vx)\vw\sigma'(\vw^\T\vx) + c(\vx)\sigma(\vw^\T\vx)\right],\vx\in\Omega \}.
    \end{equation*}
    Note that $\mA$, $\vb$, and $c$ are known functions of $\vx$.
    Each element in the above set is a function of $\bar{a}$ and $\vw$ while $\vx\in\Omega=[0,1]^d$ is a parameter. 
    Since $\norm{\vx}_\infty\leq 1$, we have
    \begin{equation*}
        \begin{aligned}
            \abs{h(\bar{a}^0_k,\vw^0_k;\vx)}
             & \leq \abs{\bar{a}_k^0}\left[M\norm{\vw_k^0}_1^3 + \frac{1}{2}M\norm{\vw_k^0}_1^3 + \frac{1}{6}M\norm{\vw_k^0}_1^3\right] \\
             & \leq 2M \abs{\bar{a}_k^0}\norm{\vw_k^0}_1^3\\
             & \leq 8 M d^3 {\left(\log \frac{4\width(d+1)}{\delta}\right)}^2.
        \end{aligned}
    \end{equation*}
    Then with probability at least $1-\delta/2$, by the Rademacher-based uniform convergence theorem, we have
    \begin{equation*}
        \begin{aligned}
            \frac{1}{\gamma \width}\sup_{\vx\in\Omega}\abs{f_{\vtheta^0}(\vx)}
             & = \sup_{\vx\in\Omega}\Abs{\frac{1}{\width}\sum_{k=1}^\width h(\bar{a}^0_k,\vw^0_k;\vx)-\Exp_{(\bar{a},\vw)\sim \mathcal{N}(0,\mI_{d+1})}h(\bar{a},\vw;\vx)}\\
             & \leq 2\Rad_{\bar{\vtheta}^0}(\fH) + 24Md^3 \left(\log\frac{4\width(d+1)}{\delta}\right)^2\sqrt{\frac{2\log(8/\delta)}{\width}},
        \end{aligned}
    \end{equation*}
    where
    \begin{align*}
        \Rad_{\bar{\vtheta}^0}(\fH) 
        &:= \frac{1}{\width}\Exp_{\vtau}\left[\sup_{\vx\in\Omega}\sum_{k=1}^\width\tau_k h(\bar{a}_k^0,\vw_k^0;\vx)\right]
        \leq I_1+I_2+I_3,\\
        I_1
        &= \frac{1}{\width}\Exp_{\vtau}\left[\sup_{\vx\in\Omega}\sum_{k=1}^\width\tau_k \bar{a}_k^0\vw^{0\T}_k\mA(\vx)\vw_k^0\sigma''(\vw^{0\T}_k\vx)\right],\\
        I_2
        &= \frac{1}{\width}\Exp_{\vtau}\left[\sup_{\vx\in\Omega}\sum_{k=1}^\width\tau_k \bar{a}_k^0\vb^\T(\vx)\vw_k^0\sigma'(\vw^{0\T}_k\vx)
        \right],\\
        I_3
        &= \frac{1}{\width}\Exp_{\vtau}\left[\sup_{\vx\in\Omega}\sum_{k=1}^\width\tau_k \bar{a}_k^0 c(\vx)\sigma(\vw^{0\T}_k\vx)
        \right],
    \end{align*}
    where $\vtau$ is a random vector in $\sN^\width$ with i.i.d. entries $\{\tau_k\}_{k=1}^\width$ following the Rademacher distribution.
    
    We only prove for $I_1$. It can be straightforwardly extended to $I_2$ and $I_3$. 
    \begin{align}
        I_1
        &= \frac{1}{\width}\Exp_{\vtau}\left[\sup_{\vx\in\Omega}\sum_{k=1}^\width\tau_k \bar{a}_k^0\vw^{0\T}_k\mA(\vx)\vw_k^0\sigma''(\vw^{0\T}_k\vx)\right]\nonumber\\
        &\leq \frac{1}{\width}\Exp_{\vtau}\left[\sup_{\vx,\vy\in\Omega}\sum_{k=1}^\width\tau_k \bar{a}_k^0\vw^{0\T}_k\mA(\vy)\vw_k^0\sigma''(\vw^{0\T}_k\vx)\right]\nonumber\\
        &= \frac{1}{\width}\Exp_{\vtau}\left[\sup_{\vx,\vy\in\Omega}\sum_{k=1}^\width\sum_{\alpha,\beta=1}^d\tau_k \bar{a}_k^0(\vw^{0\T}_k)_{\alpha}A_{\alpha\beta}(\vy)(\vw_k^0)_{\beta}\sigma''(\vw^{0\T}_k\vx)\right]\nonumber\\
        &\leq \sum_{\alpha,\beta=1}^d\frac{1}{\width}\Exp_{\vtau}\left[\sup_{\vx,\vy\in\Omega}\sum_{k=1}^\width\tau_k \bar{a}_k^0(\vw^{0\T}_k)_{\alpha}A_{\alpha\beta}(\vy)(\vw_k^0)_{\beta}\sigma''(\vw^{0\T}_k\vx)\right].\label{eq..InitialI_1Entry}
    \end{align}
    For any $\alpha,\beta\in[d]$, we have
    \begin{align}
        &~~~~\Exp_{\vtau}\left[\sup_{\vx,\vy\in\Omega}\sum_{k=1}^\width\tau_k \bar{a}_k^0(\vw^{0\T}_k)_{\alpha}A_{\alpha\beta}(\vy)(\vw_k^0)_{\beta}\sigma''(\vw^{0\T}_k\vx)\right]\nonumber\\
        &\leq \Exp_{\vtau}\left[\sup_{\vx,\vy\in\Omega}\Abs{A_{\alpha\beta}(\vy)}\Abs{\sum_{k=1}^\width\tau_k \bar{a}_k^0(\vw^{0\T}_k)_{\alpha}(\vw_k^0)_{\beta}\sigma''(\vw^{0\T}_k\vx)}\right]\nonumber\\
        &\leq M \Exp_{\vtau}\left[\sup_{\vx\in\Omega}\Abs{\sum_{k=1}^\width\tau_k \bar{a}_k^0(\vw^{0\T}_k)_{\alpha}(\vw_k^0)_{\beta}\sigma''(\vw^{0\T}_k\vx)}\right]\nonumber\\
        &\leq M \Exp_{\vtau}\left[\sup_{\vx\in\Omega}\sum_{k=1}^\width\tau_k \bar{a}_k^0(\vw^{0\T}_k)_{\alpha}(\vw_k^0)_{\beta}\sigma''(\vw^{0\T}_k\vx)\right]
        +
        M \Exp_{\vtau}\left[\sup_{\vx\in\Omega}\sum_{k=1}^\width-\tau_k \bar{a}_k^0(\vw^{0\T}_k)_{\alpha}(\vw_k^0)_{\beta}\sigma''(\vw^{0\T}_k\vx)\right]\nonumber\\
        &= 2M \Exp_{\vtau}\left[\sup_{\vx\in\Omega}\sum_{k=1}^\width\tau_k \bar{a}_k^0(\vw^{0\T}_k)_{\alpha}(\vw_k^0)_{\beta}\sigma''(\vw^{0\T}_k\vx)\right],\label{eq..InitialI_1Symmetry}
    \end{align}
    where in the third inequality, we have used the fact that $\sigma''(\vw_k^{0\T}\vx)=0$ for $\vx=0$ and for any $\vw_k^0$. Applying Lemma \ref{lem..RademacherComplexityContraction} with $\psi_k(y_k)=\bar{a}_k(\vw^{0\T}_k)_{\alpha}(\vw_k^0)_{\beta}\sigma''(y_k)$ for $k\in[\width]$, whose Lipschitz constant is $\left(\sqrt{2\log\frac{4\width(d+1)}{\delta}}\right)^3$, we have for all $\alpha,\beta\in[d]$
    \begin{equation}
        \Exp_{\vtau}\left[\sup_{\vx\in\Omega}\sum_{k=1}^\width\tau_k \bar{a}_k^0(\vw^{0\T}_k)_{\alpha}(\vw_k^0)_{\beta}\sigma''(\vw^{0\T}_k\vx)\right]
        \leq \left(\sqrt{2\log\frac{4\width(d+1)}{\delta}}\right)^3\Exp_{\vtau}\left[\sup_{\vx\in\Omega}\sum_{k=1}^\width\tau_k \vw^{0\T}_k\vx\right].\label{eq..InitialI_1Contraction}
    \end{equation}
    Therefore, combining \eqref{eq..InitialI_1Entry}, \eqref{eq..InitialI_1Symmetry}, and \eqref{eq..InitialI_1Contraction}, we obtain
    \begin{align*}
        I_1
        & \leq
        \frac{2Md^2}{\width}
        \left(\sqrt{2\log\frac{4\width(d+1)}{\delta}}\right)^3\Exp_{\vtau}\left[\sup_{\vx\in\Omega}\sum_{k=1}^\width\tau_k \vw^{0\T}_k\vx\right]\\
        & \leq \frac{2Md^3}{\sqrt{\width}}
        \left(\sqrt{2\log\frac{4\width(d+1)}{\delta}}\right)^4\sqrt{2\log(2d)}\\
        & \leq \frac{8Md^3\sqrt{2\log(2d)}}{\sqrt{\width}}
        \left(\log\frac{4\width(d+1)}{\delta}\right)^2,
    \end{align*}
    where the second inequality is by the Rademacher bound for linear predictors in Lemma \ref{lem..linear}. 
    For $I_2$ and $I_3$, we note that $\sigma(z)=\frac{1}{6}z^2\sigma''(z)$ and $\sigma'(z)=\frac{1}{2}z\sigma''(z)$. Then by a similar argument, we have
    \begin{align*}
        I_2
        &\leq \frac{4Md^2\sqrt{2\log(2d)}}{\sqrt{\width}}
        \left(\log\frac{4\width(d+1)}{\delta}\right)^2,\\
        I_3
        &\leq \frac{4Md\sqrt{2\log(2d)}}{3\sqrt{\width}}
        \left(\log\frac{4\width(d+1)}{\delta}\right)^2,\\
        \Rad_{\bar{\vtheta}^0}(\fH) 
        &\leq \frac{16Md^3\sqrt{2\log(2d)}}{\sqrt{\width}}
        \left(\log\frac{4\width(d+1)}{\delta}\right)^2.
    \end{align*}
    
    So one can get
    \begin{align*}
        \sup_{\vx\in\Omega}\abs{f_{\vtheta^0}(\vx)} &\leq 32\gamma Md^3\sqrt{\width}\sqrt{2\log(2d)}
        \left(\log\frac{4\width(d+1)}{\delta}\right)^2+ 24\gamma \sqrt{\width}Md^3\left(\log\frac{4\width(d+1)}{\delta}\right)^2\sqrt{2\log(8/\delta)}\\
        &\leq 32\gamma\sqrt{\width}Md^3\left(\log\frac{4\width(d+1)}{\delta}\right)^2\left(\sqrt{2\log(2d)}+\sqrt{2\log(8/\delta)}\right).
    \end{align*}
    Then
    \begin{equation*}
        \begin{aligned}
            \RS(\vtheta^0)
             & \leq \frac{1}{2n}\sum_{i=1}^{n}\left(1 + \abs{f_{\vtheta^0}(\vx_i)}\right)^{2}                                                                    \\
             & \leq \frac{1}{2}\left(1 + 32\gamma\sqrt{\width}Md^3\left(\log\frac{4\width(d+1)}{\delta}\right)^2\left(\sqrt{2\log(2d)}+\sqrt{2\log(8/\delta)}\right)\right)^2,
        \end{aligned}
    \end{equation*}
    where the first inequality comes from the fact that $|f|\leq 1$ by our assumption of the PDE.
\end{proof}


The following lemma shows the positive definiteness of $\mG^{(a)}$ at initialization.

\begin{lem}\label{lem:lambda_min}
    For any $\delta\in(0,1)$, if $\width\geq\frac{256n^4M^4C_d}{\lambda_S^2\delta}$, then with probability at least $1-\delta$ over the random initialization in \eqref{eqn:init}, we have
    \begin{equation*}
        \lambda_{\min}\left(\mG^{(a)}(\vtheta^0)\right)\geq\frac{3}{4}\lambda_S,
    \end{equation*}
    where $C_d := \Exp\norm{\vw}_1^{12}<+\infty$ with $\vw\sim \mathcal{N}(\vzero,\mI_d)$.
\end{lem}
\begin{proof}
    We define $\Omega_{ij} := \{\vtheta^0 \mid \lvert \mG^{(a)}_{ij}(\vtheta^0) - \mK^{(a)}_{ij}\rvert \leq \frac{\lambda_S}{4n}\}$.
    Note that
    \begin{equation*}
        \abs{g^{(a)}(\vw_k^0;\vx_i,\vx_j)} \leq \left(M\norm{\vw_k^0}_1^3 + \frac{1}{2}M\norm{\vw_k^0}^3_1 + \frac{1}{6}M\norm{\vw_k^0}_1^3\right)^2 \leq 4M^2 \norm{\vw_k^0}^6_1.
    \end{equation*}
    So
    \begin{equation*}
        \Var\left(g^{(a)}(\vw_k^0;\vx_i,\vx_j)\right) \leq \Exp\left(g^{(a)}(\vw_k^0;\vx_i,\vx_j)\right)^2 \leq 16 M^4 \Exp\norm{\vw_k^0}^{12}_1 = 16M^4 C_d,
    \end{equation*}
    and
    \begin{equation*}
        \Var\left(\mG_{ij}^{(a)}(\vtheta^0)\right) = \frac{1}{\width^2}\sum_{k=1}^{\width}\Var\left(g^{(a)}(\vw_k^0;\vx_i,\vx_j)\right) \leq \frac{16M^4 C_d}{\width}.
    \end{equation*}
    Then the probability of the event $\Omega_{ij}$ has the lower bound:
    \begin{equation*}
        \Prob(\Omega_{ij}) \geq 1 - \frac{\Var\left(\mG_{ij}^{(a)}(\vtheta^0)\right)}{[\lambda_S/(4n)]^2} \geq 1 - \frac{256M^4n^2C_d}{\lambda_S^2\width}.
    \end{equation*}
    Thus, with probability at least $\left(1 - \frac{256M^4n^2C_d}{\lambda_S^2\width}\right)^{n^2} \geq 1 - \frac{256M^4n^4C_d}{\lambda_S^2\width}$, we have all events $\Omega_{ij}$ for $i,j\in[n]$ happen. This implies that with probability at least $1 - \frac{256M^4n^4C_d}{\lambda_S^2\width}$, we have 
    \begin{equation*}
        \norm{\mG^{(a)}(\vtheta^0) - \mK^{(a)}}_{\mathrm{F}} \leq \frac{\lambda_S}{4}
    \end{equation*}
    and
    \begin{equation*}
        \lambda_{\min}\left(\mG^{(a)}(\vtheta^0)\right) \geq \lambda_S - \norm{\mG^{(a)} (\theta^0)- \mK^{(a)}}_{\mathrm{F}} \geq \frac{3}{4}\lambda_S.
    \end{equation*}
    For any $\delta\in(0,1)$, if $\width\geq\frac{256n^4M^4 C_d}{\lambda_S^2\delta}$, then with probability at least $1-\frac{256M^4n^4C_d}{\lambda_S^2\width}\geq 1-\delta$ over the initialization $\vtheta^0$, we have $\lambda_{\min}\left(\mG^{(a)}(\vtheta^0)\right) \geq \frac{3}{4}\lambda_S$.
\end{proof}

The following lemma estimates the empirical loss dynamics before the stopping time $t^*$ in \eqref{eqn:ts}.

\begin{lem}\label{lem:exp_RS}
    For any $\delta\in(0,1)$, if $\width\geq\frac{256n^4M^4 C_d}{\lambda_S^2\delta}$, then with probability at least $1-\delta$ over the random initialization in \eqref{eqn:init}, we have for any $t\in[0, t^*)$
    \begin{equation*}
        \RS(\vtheta(t)) \leq \exp\left(-\frac{\width\lambda_S t}{n}\right)\RS(\vtheta^0).
    \end{equation*}
\end{lem}
\begin{proof}
    From Lemma~\ref{lem:lambda_min}, for any $\delta\in(0,1)$ with probability at least $1-\delta$ over initialization $\vtheta^0$ and for any $t\in[0,t^*)$ with $t^*$ defined in \eqref{eqn:ts}, we have $\vtheta(t)\in\mathcal{M}(\vtheta^0)$ defined in \eqref{eqn:Mtheta} and
    \begin{equation*}
        \begin{aligned}
            \lambda_{\min}\left(\mG^{(a)}(\vtheta)\right)
             & \geq \lambda_{\min}\left(\mG^{(a)}(\vtheta^0)\right) - \norm{\mG^{(a)}(\vtheta) - \mG^{(a)}(\vtheta^0)} _{\mathrm{F}} \\
             & \geq \frac{3}{4}\lambda_S - \frac{1}{4}\lambda_S                                                          \\
             & = \frac{1}{2}\lambda_S.
        \end{aligned}
    \end{equation*}
    Note that $\mG_{ij} = \frac{1}{\width}\nabla_{\vtheta}f_{\vtheta}(\vx_i)\cdot\nabla_{\vtheta}f_{\vtheta}(\vx_j)$ and $\nabla_{\vtheta}\RS = \frac{1}{n}\sum_{i=1}^{n}e_i\nabla_{\vtheta}f_{\vtheta}(\vx_i)$, so
    \begin{equation*}
        \norm{\nabla_{\vtheta}\RS (\vtheta(t))}^2_2 = \frac{\width}{n^2}\ve^{\T}\mG(\vtheta(t)) \ve \geq \frac{\width}{n^2}\ve^{\T}\mG^{(a)}(\vtheta(t))\ve,
        \end{equation*}
        where the last equation is true by the fact that $G^{(w)}(\vtheta(t))$ is a Gram matrix and hence positive semi-definite. Together with 
        \begin{equation*}
         \frac{\width}{n^2}\ve^{\T}\mG^{(a)}(\vtheta(t))\ve\geq \frac{2\width}{n}\lambda_{\min}\left(\mG^{(a)}(\vtheta(t))\right)\RS(\vtheta(t))\geq\frac{\width}{n}\lambda_S\RS(\vtheta(t)),
    \end{equation*}
    then finally we get
    \begin{equation*}
        \frac{\D}{\D t}\RS(\vtheta(t)) = - \norm{\nabla_{\vtheta}\RS(\vtheta(t))}^2_2 \leq - \frac{\width}{n}\lambda_S\RS(\vtheta(t)).
    \end{equation*}
    Integrating the above equation yields the conclusion in this lemma.
\end{proof}

The following lemma shows that the parameters in the two-layer neural network is uniformly bounded in time during the training before time $t^*$.

\begin{lem}\label{prop:a_w}
    For any $\delta\in(0,1)$, if $\width\geq\max\left\{\frac{512n^4M^4 C_d}{\lambda_S^2\delta}, \frac{200\sqrt{2}Md^3n\log(4\width(d+1)/\delta)\sqrt{\RS(\vtheta^0)}}{\lambda_S}\right\}$, then with probability at least $1-\delta$ over the random initialization in \eqref{eqn:init}, for any $t\in[0, t^\ast)$ and any $k\in [\width]$,
    \begin{equation*}
        \begin{aligned}
             & \abs{a_k(t) - a_k(0)} \leq q,                    & \norm{\vw_k(t) - \vw_k(0)}_\infty \leq q,\\
             & \abs{a_k(0)} \leq \gamma\eta, \quad 
             & \norm{\vw_k(0)}_\infty\leq \eta,
        \end{aligned}
    \end{equation*}
    where
    \begin{equation*}
        q := \frac{320 M d^3 (\log\frac{4\width(d+1)}{\delta})^{3/2}n\sqrt{\RS(\vtheta^0)}}{\width\lambda_S}
    \end{equation*}
    and
    \[
    \eta:=\sqrt{2\log\frac{4\width(d+1)}{\delta}}.
    \]
\end{lem}
\begin{proof}
    Let $\xi(t) = \max\limits_{k\in[\width],s\in[0,t]}\{\abs{a_k(s)},\norm{\vw_k(s)}_\infty\}$. Note that
    \begin{align*}
        \abs{\nabla_{a_k}\RS}^2
         & = \left\{\frac{1}{n}\sum_{i=1}^n e_i\left[\vw_k^{\T}\mA(\vx_i)\vw_k\sigma''(\vw_k^\T\vx_i) + \vb^{\T}(\vx_i)\vw_k\sigma'(\vw_k^\T\vx_i) + c(\vx_i)\sigma(\vw_k^\T\vx_i)\right]\right\}^2\\
         & \leq 8 M^2 \norm{\vw_k}_1^6\RS(\vtheta)\\
         & \leq 8 M^2d^6 (\xi(t))^6\RS(\vtheta),
    \end{align*}
    and
    \begin{align*}
        \norm{\nabla_{\vw_k}\RS}^2_\infty
        &=     \Big\lVert\frac{1}{n}\sum_{i=1}^n e_i a_k\Big[2\mA(\vx_i)\vw_k\sigma''(\vw_k^\T\vx_i) + \vw_k^{\T}\mA(\vx_i)\vw_k\sigma^{(3)}(\vw_k^\T\vx_i)\vx_i \\
        &~~~~+ \sigma'(\vw_k^\T\vx_i)\vb(\vx_i) + \vb^{\T}(\vx_i)\vw_k\sigma''(\vw_k^\T\vx_i)\vx_i + c(\vx_i)\sigma'(\vw_k^\T\vx_i)\vx_i\Big]\Big\rVert_\infty^2   \\
        &\leq \abs{a_k}^2 2\RS(\vtheta)\Big(2 M\norm{\vw_k}^2_1 + M\norm{\vw_k}^2_1 + \frac{1}{2}M\norm{\vw_k}^2_1 + M\norm{\vw_k}^2_1 + M \frac{1}{2}\norm{\vw_k}^2_1\Big)^2\\
        &\leq 50 M^2 \norm{\vw_k}^4_1\abs{a_k}^2\RS(\vtheta)\\
        &\leq 50 M^2d^4(\xi(t))^6\RS(\vtheta).
    \end{align*}
    From Lemma~\ref{lem:exp_RS}, if $\width\geq \frac{512M^4 n^4C_d}{\lambda_s^2\delta}$, then with probability at least $1 - \delta/2$ over initialization
    \begin{equation*}
        \begin{aligned}
            \abs{a_k(t) - a_k(0)}
             & \leq \int_0^t\abs{\nabla_{a_k}\RS(\vtheta(s))}\diff{s}                                                                 \\
             & \leq 2\sqrt{2}Md^3 \int_{0}^{t} \xi^3(t)\sqrt{\RS(\vtheta(s))}\diff{s}                                        \\
             & \leq 2\sqrt{2}Md^3 \xi^3(t)\int_{0}^{t}\sqrt{\RS(\vtheta^0)}\exp\left(-\frac{\width\lambda_S s}{2n}\right)\diff{s} \\
             & \leq \frac{4\sqrt{2}Md^3 n\sqrt{\RS(\vtheta^0)}}{\width\lambda_S}\xi^3(t)                                           \\
             & \leq p\xi^3(t),
        \end{aligned}
    \end{equation*}
    where $p := \frac{10\sqrt{2}d^3 M n\sqrt{\RS(\vtheta^0)}}{\width\lambda_S}$. Similarly,
    \begin{equation*}
        \begin{aligned}
            \norm{\vw_k(t) - \vw_k(0)}_\infty
             & \leq \int_{0}^{t} \norm{\nabla_{\vw_k}\RS(\vtheta(s))}_\infty\diff{s}                                                                      \\
             & \leq 5\sqrt{2}Md^2 \int_{0}^t \xi^3(t)\sqrt{\RS(\vtheta(s))} \diff{s}                                           \\
             & \leq 5\sqrt{2}Md^2 \xi^3(t) \int_{0}^{t} \sqrt{\RS(\vtheta^0)}\exp\left(-\frac{\width\lambda_S s}{2n}\right)\diff{s} \\
             & \leq \frac{10\sqrt{2}Md^2 n\sqrt{\RS(\vtheta^0)}}{\width\lambda_S}\xi^3(t)                                             \\
             & \leq p\xi^3(t).
        \end{aligned}
    \end{equation*}
    So
    \begin{equation}\label{eqn:xib}
        \xi(t) \leq \xi(0) + p\xi^3(t).
    \end{equation}
    From Lemma~\ref{lem1} with probability at least $1 - \delta/2$,
    \begin{align}\label{eqn:xi0}
        \xi(0)=\max_{k\in[\width]}\{\abs{a_k(0)},\norm{\vw_k(0)}_\infty\}
        &\leq \max\left\{\gamma\sqrt{2\log\frac{4\width(d+1)}{\delta}},\sqrt{2\log\frac{4\width(d+1)}{\delta}}\right\}\nonumber\\
        &\leq \sqrt{2\log\frac{4\width(d+1)}{\delta}}=\eta.
    \end{align}
    Since
    \begin{equation*}
        \width
        \geq \frac{200\sqrt{2}Md^3n\log(4\width(d+1)/\delta)\sqrt{\RS(\vtheta^0)}}{\lambda_S}
        = 10\width p\eta^2,
    \end{equation*}
    then $p\leq\frac{1}{10}\left(2\log\frac{4\width(d+1)}{\delta}\right)^{-1} = \frac{1}{10}\eta^{-2}$ and $p(2\eta)^2\leq \frac{2}{5}$. Let
    \begin{equation*}
        t_0 := \inf\{t \mid \xi(t) > 2\eta\}.
    \end{equation*}
    We will prove $t_0\geq t^*$ by contradiction. Suppose that $t_0 < t^*$. For $t\in[0, t_0)$, by \eqref{eqn:xib}, \eqref{eqn:xi0}, and $\xi(t)\leq 2\eta$, we have 
    \begin{equation*}
        \xi(t) \leq \eta + p(2\eta)^2\xi(t)
        \leq \eta + \frac{2}{5}\xi(t),
    \end{equation*}
    then
    \begin{equation*}
        \xi(t) \leq \frac{5}{3}\eta.
    \end{equation*}
    After letting $t\to t_0$, the inequality just above contradicts with the definition of $t_0$. So $t_0 \geq  t^*$ and then $\xi(t) \leq 2\eta$ for all $t \in[0,t^*)$. Thus
    \begin{equation*}
        \begin{aligned}
            \abs{a_k(t) - a_k(0)}        & \leq8 \eta^3 p            \\
            \norm{\vw_k(t) - \vw_k(0)}_\infty & \leq8\eta^3 p.
        \end{aligned}
    \end{equation*}
    Finally, notice that
    \begin{equation}\label{eqn:q}
        \begin{aligned}
            8\eta^3 p
             & = 8\sqrt{8}\left(\log \frac{4\width(d+1)}{\delta}\right)^{3/2}\frac{10\sqrt{2}Md^3 n\sqrt{\RS(\vtheta^0)}}{\width\lambda_S} \\
             & = \frac{320M d^3 \left(\log \frac{4\width(d+1)}{\delta}\right)^{3/2} n\sqrt{\RS(\vtheta^0)}}{\width\lambda_S}                              \\
             & = q,
        \end{aligned}
    \end{equation}
    which ends the proof.
\end{proof}

\subsection{Proof of Theorem \ref{thm:lcr}}
\label{sec:thmlcr}

\begin{proof}[Proof of Theorem \ref{thm:lcr}.]
    From Lemma~\ref{lem:exp_RS}, it is sufficient to prove that the stopping time $t^*$ in Lemma~\ref{lem:exp_RS} is equal to $+\infty$. We will prove this by contradiction.
    
    Suppose $t^* < +\infty$. Note that
    \begin{equation}\label{eqn:lcr1}
        \abs{\mG_{ij}^{(a)}(\vtheta(t^*)) - \mG_{ij}^{(a)}(\vtheta(0))} \leq \frac{1}{\width}\sum_{k=1}^\width \abs{g(\vw_k(t^*);\vx_i,\vx_j) - g(\vw_k(0);\vx_i,\vx_j)}.
    \end{equation}
    By the mean value theorem,
    \begin{align*}
        \abs{g(\vw_k(t^*);\vx_i,\vx_j) - g(\vw_k(0);\vx_i,\vx_j)} 
        &\leq \norm{\nabla g\left(c\vw_k(t^*) + (1-c)\vw_k(0);\vx_i,\vx_j\right)}_\infty\norm{\vw_k(t^*) - \vw_k(0)}_1
    \end{align*}
    for some $c\in (0, 1)$. Further computation yields
    \begin{equation*}
        \begin{aligned}
            \nabla g(\vw;\vx_i,\vx_j) =
             & \Big[2\mA(\vx_i)\vw\sigma''(\vw^\T\vx_i) + \vw^{\T}\mA(\vx_i)\vw\sigma^{(3)}(\vw^\T\vx_i)\vx_i + \sigma'(\vw^\T\vx_i)\vb(\vx_i)      \\
             & \quad\quad + \vb^{\T}(\vx_i)\vw\sigma''(\vw^\T\vx_i)\vx_i + c(\vx_i)\sigma'(\vw^\T\vx_i)\vx_i\Big]                                       \\
             & \quad \times \Big[\vw^{\T}\mA(\vx_j)\vw\sigma''(\vw^\T\vx_j) + \vb^{\T}(\vx_j)\vw\sigma'(\vw^\T\vx_j) + c(\vx_j)\sigma(\vw^\T\vx_j)\Big]  \\
             & + \Big[2\mA(\vx_j)\vw\sigma''(\vw^\T\vx_j) + \vw^{\T}\mA(\vx_j)\vw\sigma^{(3)}(\vw^\T\vx_j)\vx_i + \sigma'(\vw^\T\vx_i)\vb(\vx_i)    \\
             & \quad \quad  + \vb^{\T}(\vx_j)\vw\sigma''(\vw^\T\vx_j)\vx_j + c(\vx_j)\sigma'(\vw^\T\vx_j)\vx_j\Big]                                     \\
             & \quad \times \Big[\vw^{\T}\mA(\vx_i)\vw\sigma''(\vw^\T\vx_i) + \vb^{\T}(\vx_i)\vw\sigma'(\vw^\T\vx_i) + c(\vx_i)\sigma(\vw^\T\vx_i)\Big]
        \end{aligned}
    \end{equation*}
    for all $\bm{w}$. Hence, it holds for all $\bm{w}$ that 
    \begin{equation*}
        \begin{aligned}
            \norm{\nabla g(\vw;\vx_i,\vx_j)}_\infty
             & \leq 2 \Big[2M\norm{\vw}^2_1 + M\norm{\vw}^2_1 + \frac{1}{2}M\norm{\vw}_1^2  + M\norm{\vw}^2_1  + \frac{1}{2}M\norm{\vw}^2_1\Big] \\
             & \quad \times \Big[M\norm{\vw}^3_1 + \frac{1}{2}M\norm{\vw}^3_1 + \frac{1}{6} M\norm{\vw}^3_1 \Big]                                         \\
             & \leq 2 (5M\norm{\vw}^2_1)(2M\norm{\vw}^3_1 )                                                                                                 \\
             & = 20 M^2\norm{\vw}^5_1.
        \end{aligned}
    \end{equation*}
    Therefore, the bound in \eqref{eqn:lcr1} becomes
    \begin{equation}\label{eqn:lcr2}
        \abs{\mG_{ij}^{(a)}(\vtheta(t^*)) - \mG_{ij}^{(a)}(\vtheta(0))} \leq \frac{20 M^2}{\width}\sum_{k=1}^\width \norm{c\vw_k(t^*) + (1-c)\vw_k(0)}^5_1 \norm{\vw_k(t^*) - \vw_k(0)}_1.
    \end{equation}
    By Lemma~\ref{prop:a_w},
    \begin{equation*}
        \norm{c\vw_k(t^*) + (1-c)\vw_k(0)}_1 \leq \norm{\vw_k(0)}_1 + \norm{\vw_k(t^*) - \vw_k(0)}_1 \leq d(\eta + q) \leq 2d\eta,
    \end{equation*}
    where $\eta$ and $q$ are defined in Lemma~\ref{prop:a_w}. So, \eqref{eqn:lcr2} and the above inequalities indicate
    \begin{equation*}
        \abs{\mG_{ij}^{(a)}(\vtheta(t^*)) - \mG_{ij}^{(a)}(\vtheta(0))} \leq 20 M^2 (2d\eta)^5 d q = 640 M^2 d^6 \eta^5 q,
    \end{equation*}
    and
    \begin{equation*}
        \begin{aligned}
            \norm{\mG^{(a)}(\vtheta(t^*)) - \mG^{(a)}(\vtheta(0))} _{\mathrm{F}}
             & \leq 640M^2 d^6 n\eta^5 q                                                                        \\
             & < \frac{2^{21}M^3 d^9n^2(\log\frac{4\width(d+1)}{\delta})^{4}\sqrt{\RS(\vtheta^0)}}{\width\lambda_S} \\
             & \leq \frac{1}{4}\lambda_S,
        \end{aligned}
    \end{equation*}
    if we choose
    \begin{equation*}
        \width\geq\frac{2^{23}M^3d^9n^2(\log(4\width(d+1)/\delta))^{4}\sqrt{\RS(\vtheta^0)}}{\lambda_S^2}.
    \end{equation*}
    The fact that $\norm{\mG^{(a)}(\vtheta(t^*)) - \mG^{(a)}(\vtheta(0))} _{\mathrm{F}}\leq \frac{1}{4}\lambda_S$ above contradicts with the definition of $t^*$ in \eqref{eqn:ts}. Hence, we have completed the proof.
\end{proof}

\section{A priori Estimates of Generalization Error for Two-layer Neural Networks}\label{sec:gen_err}
To obtain good generalization, instead of minimizing $\RS$, we minimize the regularized risk of $\RS(\vtheta)$:
\begin{equation}
    J_{S,\lambda}(\vtheta)
      :=\RS(\vtheta)+\frac{\lambda}{\sqrt{n}}\norm{\vtheta}_{\fP}^3
     \end{equation}
to obtain     
     \begin{equation}
    \vtheta_{S,\lambda}
     = \arg\min_{\vtheta} J_{S,\lambda}(\vtheta).
\end{equation}
Our work is inspired by the seminal work in \cite{Weinan2019,e2019priori} and the proof is a variant of the proof therein. But as we shall see, the differential operator increases the technical difficulty in the analysis: extra non-linearity in the parameters, which makes existing mean field analysis \cite{MeiE7665} not applicable. We will use the path norm defined in Definition \ref{prop:2L} adaptive to the PDE problem, instead of using the path norm in \cite{Weinan2019,e2019priori} for regression problems. We will show that the PDE solution network $\phi(\vx;\vtheta_{S,\lambda})$ generalize well if the true solution is in the Barron-type space defined in Definition \ref{def:bfun}, which is also a variance of the Barron-type space in \cite{Weinan2019,e2019priori}. The generalization error is measured in terms of how well $f(\vx;\vtheta_{S,\lambda}):=\fL \phi(\vx;\vtheta_{S,\lambda})\approx f(\vx)$ generalizes from the random training samples $S=\{\vx_i\}_{i=1}^n\subset\Omega$ to arbitrary samples in $\Omega$.

Recall that $f(\vx;\vtheta)$, also denoted as $f_{\vtheta}(\vx)$, is the result of the differential operator $\fL$ acting on a two-layer neural network $\phi(\vx;\vtheta)$ in the domain $\Omega$. In fact, $f(\vx;\vtheta)$ is also a two-layer neural network as explained in \eqref{eqn:NN2}. Hence, the generalization error analysis of deep learning-based PDE solvers is reduced to the generalization analysis of the special two-layer neural network $f(\vx;\vtheta)$ fitting $f(\vx)$. The special structure of $f(\vx;\vtheta)$ leads to significant difficulty in analyzing the generalization error compared to traditional two-layer neural networks in the literature.

We will first summarize and prove several lemmas related to Rademacher complexity in Section \ref{sec:preRad}. The proofs of our main theorems for the generalization bound in Theorems \ref{thm:gen1} and \ref{thm:gen2} are presented in Section \ref{sec:gen}.

\subsection{Preliminary Lemmas of Rademacher Complexity}
\label{sec:preRad}

First, we define the set of functions 
\[
\fF_Q = \{f(\vx;\vtheta)= \sum_{k=1}^\width a_k[\vw_k^\T\mA(\vx)\vw_k\sigma''(\vw_k^\T\vx)+\vb^\T(\vx)\vw_k\sigma'(\vw_k^\T\vx)+c(\vx)\sigma(\vw_k^\T\vx)]\mid \norm{\vtheta}_{\fP}\leq Q\}.
\]

Second, we estimate the Rademacher complexity of the class of special two-layer neural networks $\fF_Q$.

\begin{lem}[Rademacher complexity of two-layer neural networks]\label{lem..RademacherComplexityTwoLayerPathNorm}
    The Rademacher complexity of $\fF_Q$ over a set of $n$ uniform distributed random samples of $\Omega$, denoted as $S=\{\vx_1,\dots,\vx_n\}$, has an upper bound
    \begin{equation*}
        \Rad_S(\fF_Q)\leq \frac{4MQd^2\sqrt{2\log(2d)}}{\sqrt{n}},
    \end{equation*}
    where $M$ is the upper bound of the differential operator $\fL$ introduced in \eqref{eqn:introM}.
\end{lem}
\begin{proof}
    Let $\hat{\vw}_k=\vw_k/\norm{\vw_k}_1$ for $k=1,\cdots,\width$ and $\vtau$ be a random vector in $\sN^d$ with i.i.d. entries following the Rademacher distribution. Then
    \begin{align}
         & ~~~~n\Rad_S(\fF_Q)  \nonumber\\
         & = \Exp_{\vtau}\left\{\sup_{\norm{\vtheta}_{\fP}\leq Q}\sum_{i=1}^n\tau_i\sum_{k=1}^\width a_k[\vw_k^\T\mA(\vx_i)\vw_k\sigma''(\vw_k^\T\vx_i)+\vb^\T(\vx_i)\vw_k\sigma'(\vw_k^\T\vx_i)+c(\vx_i)\sigma(\vw_k^\T\vx_i)]\right\} \nonumber\\
         & \leq \Exp_{\vtau}\left[\sup_{\norm{\vtheta}_{\fP}\leq Q}\sum_{i=1}^n\tau_i\sum_{k=1}^\width a_k \vw_k^\T\mA(\vx_i)\vw_k\sigma''(\vw_k^\T\vx_i) \right]
        +\Exp_{\vtau}\left[\sup_{\norm{\vtheta}_{\fP}\leq Q}\sum_{i=1}^n\tau_i\sum_{k=1}^\width a_k \vb^\T(\vx_i)\vw_k\sigma'(\vw_k^\T\vx_i) \right]                                                                                    \nonumber\\
         & ~~+\Exp_{\vtau}\left[\sup_{\norm{\vtheta}_{\fP}\leq Q}\sum_{i=1}^n\tau_i\sum_{k=1}^\width a_k c(\vx_i)\sigma(\vw_k^\T\vx_i) \right]
        \nonumber\\
         & =:I_1+I_2+I_3.\label{eq..RademacherComplexityDecompositionI1I2I3}
    \end{align}
    We first estimate $I_1$ as follows
    \begin{align}
        I_1
         & = \Exp_{\vtau}\left[\sup_{\norm{\vtheta}_{\fP}\leq Q}\sum_{i=1}^n\tau_i\sum_{k=1}^\width a_k\norm{\vw_k}_1^3\hat{\vw}_k^\T\mA(\vx_i)\hat{\vw}_k\sigma''(\hat{\vw}_k^\T\vx_i)\right]                         \nonumber\\
         & \leq \Exp_{\vtau}\left[\sup_{\norm{\vtheta}_{\fP}\leq Q,\norm{\vu_k}_1=1,\forall k}\sum_{i=1}^n\tau_i\sum_{k=1}^\width a_k\norm{\vw_k}_1^3\vu_k^\T\mA(\vx_i)\vu_k\sigma''(\vu_k^\T\vx_i)\right] \nonumber\\
         & \leq \Exp_{\vtau}\left[\sup_{\norm{\vtheta}_{\fP}\leq Q,\norm{\vu_k}_1=1,\forall k}\sum_{k=1}^\width \Abs{a_k\norm{\vw_k}_1^3}\Abs{\sum_{i=1}^n \tau_i\vu_k^\T\mA(\vx_i)\vu_k\sigma''(\vu_k^\T\vx_i)}\right]   \nonumber\\
         & = \Exp_{\vtau}\left[\sup_{\norm{\vtheta}_{\fP}\leq Q,\norm{\vu}_1=1}\sum_{k=1}^\width \abs{a_k}\norm{\vw_k}_1^3\Abs{\sum_{i=1}^n \tau_i\vu^\T\mA(\vx_i)\vu\sigma''(\vu^\T\vx_i)}\right]   \nonumber\\
         & \leq Q\Exp_{\vtau}\left[\sup_{\norm{\vu}_1\leq 1,\norm{\vp}_1\leq 1,\norm{\vq}_1\leq 1}\Abs{\sum_{i=1}^n \tau_i\vp^\T\mA(\vx_i)\vq\sigma''(\vu^\T\vx_i)}\right]\nonumber\\
         & = Q\Exp_{\vtau}\left[\sup_{\norm{\vu}_1\leq 1,\norm{\vp}_1\leq 1,\norm{\vq}_1\leq 1}\Abs{\vp^\T\left(\sum_{i=1}^n \tau_i\mA(\vx_i)\sigma''(\vu^\T\vx_i)\right)\vq}\right]\nonumber\\
         & = Q\Exp_{\vtau}\left[\sup_{\norm{\vu}_1\leq 1,\norm{\vp}_1\leq 1,\norm{\vq}_1\leq 1}\sum_{\alpha,\beta=1}^d\abs{p_\alpha}\abs{q_\beta}\Abs{\sum_{i=1}^n \tau_i A_{\alpha\beta}(\vx_i)\sigma''(\vu^\T\vx_i)}\right]\nonumber\\
         & \leq Q\Exp_{\vtau}\left[\sup_{\norm{\vu}_1\leq 1}\max_{\alpha,\beta\in[d]}\Abs{\sum_{i=1}^n \tau_i A_{\alpha\beta}(\vx_i)\sigma''(\vu^\T\vx_i)}\right]\nonumber\\
         & \leq Q\Exp_{\vtau}\left[\sup_{\norm{\vu}_1\leq 1}\sum_{\alpha,\beta=1}^{d}\Abs{\sum_{i=1}^n \tau_i A_{\alpha\beta}(\vx_i)\sigma''(\vu^\T\vx_i)}\right]\nonumber\\
         & \leq Q\Exp_{\vtau}\left[\sum_{\alpha,\beta=1}^{d}\sup_{\norm{\vu}_1\leq 1}\Abs{\sum_{i=1}^n \tau_i A_{\alpha\beta}(\vx_i)\sigma''(\vu^\T\vx_i)}\right]\nonumber\\
         & = Q\sum_{\alpha,\beta=1}^{d}\Exp_{\vtau}\left[\sup_{\norm{\vu}_1\leq 1}\Abs{\sum_{i=1}^n \tau_i A_{\alpha\beta}(\vx_i)\sigma''(\vu^\T\vx_i)}\right].\label{eq..I_1PathNorm}
    \end{align}
    Note that $\sigma''(\vu^\T\vx_i)=0$ for $\vu=0$ and for any $\vx_i$. For any $\alpha,\beta\in[d]$, we have
    \begin{align}
        \Exp_{\vtau}\left[\sup_{\norm{\vu}_1\leq 1}\Abs{\sum_{i=1}^n \tau_i A_{\alpha\beta}(\vx_i)\sigma''(\vu^\T\vx_i)}\right]
        & \leq \Exp_{\vtau}\left[\sup_{\norm{\vu}_1\leq 1}\sum_{i=1}^n \tau_i A_{\alpha\beta}(\vx_i)\sigma''(\vu^\T\vx_i)\right]\nonumber\\
        &~~~~+\Exp_{\vtau}\left[\sup_{\norm{\vu}_1\leq 1}\sum_{i=1}^n -\tau_i A_{\alpha\beta}(\vx_i)\sigma''(\vu^\T\vx_i)\right]\nonumber\\
        & = 2\Exp_{\vtau}\left[\sup_{\norm{\vu}_1\leq 1}\sum_{i=1}^n \tau_i A_{\alpha\beta}(\vx_i)\sigma''(\vu^\T\vx_i)\right].\label{eq..I_1Symmetry}
    \end{align}
    Applying Lemma \ref{lem..RademacherComplexityContraction} with $\psi_i(y_i)=A_{\alpha\beta}(\vx_i)\sigma''(y_i)$ for $i\in[n]$, whose Lipschitz constant is $M$, we have for all $\alpha,\beta\in[d]$
    \begin{equation}
        \Exp_{\vtau}\left[\sup_{\norm{\vu}_1\leq 1}\sum_{i=1}^n \tau_i A_{\alpha\beta}(\vx_i)\sigma''(\vu^\T\vx_i)\right]
        \leq M\Exp_{\vtau}\left[\sup_{\norm{\vu}_1\leq 1}\sum_{i=1}^n \tau_i \vu^\T\vx_i\right].\label{eq..I_1Contraction}
    \end{equation}
    Therefore, combining \eqref{eq..I_1PathNorm}, \eqref{eq..I_1Symmetry}, and \eqref{eq..I_1Contraction}, we obtain
    \begin{align*}
        I_1
        & \leq 2MQd^2\Exp_{\vtau}\left[\sup_{\norm{\vu}_1\leq 1}\sum_{i=1}^n \tau_i \vu^\T\vx_i\right]\\
        & \leq 2MQd^2\sqrt{n}\sqrt{2\log(2d)},
    \end{align*}
    where the last inequality comes from the Rademacher bound for linear predictors in Lemma \ref{lem..linear}.

    For $I_2$ and $I_3$, we note that $\sigma(z)=\frac{1}{6}z^2\sigma''(z)$ and $\sigma'(z)=\frac{1}{2}z\sigma''(z)$. Then by similar arguments, we have
    \begin{align*}
        I_2
        &\leq MQd\sqrt{n}\sqrt{2\log(2d)},\\
        I_3
        &\leq \frac{1}{3}MQ\sqrt{n}\sqrt{2\log(2d)}.
    \end{align*}
    These estimates for $I_1,I_2,I_3$ combined with \eqref{eq..RademacherComplexityDecompositionI1I2I3} complete the proof.
\end{proof}

\subsection{Proofs of Generalization Bounds}
\label{sec:gen}

In the proofs of this section, we will first show in Proposition \ref{prop..ApproximationErrorTwoLayerBarronNorm} that two-layer neural networks $f(\vx;\vtheta)$ in \eqref{eqn:NN2} can approximate Barron-type functions with an approximation error $O\left( \frac{\norm{f}_{\fB}^2}{\width}\right)$. Second, for an arbitrary $f(\vx;\vtheta)=\fL \phi(\vx;\vtheta)$, we show its a posteriori generalization bound $\abs{\RD(\vtheta)-\RS(\vtheta)}\leq O\left(\frac{\norm{\vtheta}_{\fP}^2\log\norm{\vtheta}_{\fP}}{\sqrt{n}}\right)$ in Theorem \ref{thm:gen1}. Finally, the a priori generalization bound $\RD(\vtheta_{S,\lambda}) \leq O\left( \frac{\norm{f}_{\fB}^2}{\width}+\frac{\norm{f}_{\fB}^2\log\norm{f}_{\fB}}{\sqrt{n}}\right)$ is proved in Theorem \ref{thm:gen2}, where the first and second terms comes from the approximation error bound and the a posteriori generalization bound.

First, the approximation capacity of two-layer neural networks $f(\vx;\vtheta)$ can be characterized by Proposition \ref{prop..ApproximationErrorTwoLayerBarronNorm} below.

\begin{prop}[Approximation Error]\label{prop..ApproximationErrorTwoLayerBarronNorm}
    For any $f\in \fB(\Omega)$, there exists a two-layer neural network $f(\vx; \tilde{\vtheta})$
    of width $\width$ with $\norm{\tilde{\vtheta}}_{\fP}\leq 2\norm{f}_\fB$,
    \begin{equation*}
        \RD(\tilde{\vtheta}) := \Exp_{\vx\sim\fD}\tfrac{1}{2}(f(\vx, \tilde{\vtheta}) - f(\vx))^2
        \leq \frac{6M^2\norm{f}_{\fB}^2}{\width},
    \end{equation*}
    where $M$ introduced in \eqref{eqn:introM} controls the upper bound of the differential operator and $\width$ is the width of the neural network.
\end{prop}
\begin{proof}
    Without loss of generality, let $\rho$ be the best representation, i.e., $\norm{f}_{\fB}^2= \Exp_{(a,\vw)\sim\rho}
        \abs{a}^2\norm{\vw}_1^6$. We set $\bar{\vtheta} = \{\frac{1}{\width}a_k,\vw_k\}_{k=1}^\width$, where $(a_k,\vw_k)$, $k=1,\cdots,\width$ are independent sampled from $\rho$. Let
    \begin{align*}
        f_{\bar{\vtheta}}(\vx)
         = \frac{1}{\width}\sum_{k=1}^\width a_k[\vw_k^\T\mA(\vx)\vw_k\sigma''(\vw_k^\T\vx)+\vb^\T(\vx)\vw_k\sigma'(\vw_k^\T\vx)+c(\vx)\sigma(\vw_k^\T\vx)].
    \end{align*}
    Recall the definition $\RD(\bar{\vtheta})=\Exp_{\vx\sim\fD}\frac{1}{2}\abs{f_{\bar{\vtheta}}(\vx)-f(\vx)}^2$.
    Then
    \begin{align*}
         & ~~~~2\Exp_{\bar{\vtheta}}\RD(\bar{\vtheta})\\
         & =\Exp_{\vx\sim\fD}\Exp_{\bar{\vtheta}}\abs{f_{\bar{\vtheta}}(\vx)-f(\vx)}^2\\
         & =\Exp_{\vx\sim\fD}\mathrm{Var}_{\{(a_k,\vw_k)\} \text{i.i.d.}\sim\rho}\left(\frac{1}{\width}\sum_{k=1}^\width a_k[\vw_k^\T\mA(\vx)\vw_k\sigma''(\vw_k^\T\vx)+\vb^\T(\vx)\vw_k\sigma'(\vw_k^\T\vx)+c(\vx)\sigma(\vw_k^\T\vx)]\right) \\
         & =\Exp_{\vx\sim\fD}\frac{1}{\width}\mathrm{Var}_{(a,\vw)\sim\rho}\left(a[\vw^\T\mA(\vx)\vw\sigma''(\vw^\T\vx)+\vb^\T(\vx)\vw\sigma'(\vw^\T\vx)+c(\vx)\sigma(\vw^\T\vx)]\right)\\
         & \leq \frac{1}{\width}\Exp_{\vx\sim\fD}\Exp_{(a,\vw)\sim\rho}\left(a[\vw^\T\mA(\vx)\vw\sigma''(\vw^\T\vx)+\vb^\T(\vx)\vw\sigma'(\vw^\T\vx)+c(\vx)\sigma(\vw^\T\vx)]\right)^2\\
         & \leq \frac{1}{\width}\Exp_{\vx\sim\fD}\Exp_{(a,\vw)\sim\rho}\abs{a}^2\left(M\norm{\vw}_1^3+\tfrac{1}{2}M\norm{\vw}_1^3+\tfrac{1}{6}M\norm{\vw}_1^3\right)^2\\
         & \leq \frac{4M^2}{\width}\Exp_{(a,\vw)\sim\rho}\abs{a}^2\norm{\vw}_1^6\\
         & = \frac{4M^2\norm{f}_{\fB}^2}{\width}.
    \end{align*}
    Also, we have
    \begin{align*}
        \Exp_{\bar{\vtheta}}\norm{\bar{\vtheta}}_{\fP}
         & = \Exp_{\{(a_k,\vw_k)\} \text{i.i.d.}\sim\rho}\frac{1}{\width}\sum_{k=1}^\width\abs{a_k}\norm{\vw_k}_1^3 \\
         & = \Exp_{(a,\vw)\sim\rho}\abs{a}\norm{\vw}_1^3\\
         & \leq \norm{f}_{\fB}.
    \end{align*}
    Define two events $E_1:=\{\RD(\bar{\vtheta})<\frac{6M^2\norm{f}_{\fB}^2}{\width}\}$ and $E_2:=\{\norm{\bar{\vtheta}}_{\fP}<2\norm{f}_{\fB}\}$.
    By Markov inequality, we have
    \begin{align*}
        \Prob(E_1)
         & =1-\Prob\left(\RD(\bar{\vtheta})\geq\frac{6M^2\norm{f}_{\fB}^2}{\width}\right)
        \geq1-\frac{\Exp_{\bar{\vtheta}}\RD(\bar{\vtheta})}{6M^2\norm{f}_{\fB}^2/\width}
        \geq \frac{2}{3},                                             \\
        \Prob(E_2)
         & =1-\Prob(
        \norm{\bar{\vtheta}}_{\fP}\geq2\norm{f}_{\fB})
        \geq 1-\frac{\Exp_{\bar{\vtheta}}\norm{\bar{\vtheta}}_{\fP}}{2\norm{f}_{\fB}}
        \geq \frac{1}{2}.
    \end{align*}
    Thus
    \begin{equation*}
        \Prob(E_1\cap E_2)\geq \Prob(E_1)+\Prob(E_2)-1\geq \frac{2}{3}+\frac{1}{2}-1>0.
    \end{equation*}
\end{proof}

Second, we use Theorem \ref{thm..RademacherComplexityGeneralizationGap} with $\mathcal{F}=\mathcal{F}_Q$ and $\mathcal{Z}=\Omega$ to show the a posteriori generalization bound in Theorem \ref{thm:gen1}.

\begin{proof}[Proof of Theorem \ref{thm:gen1}.]
    Let $\fH_Q:=\{\ell(f(\vx),f_{\vtheta} (\vx))\mid\norm{\vtheta}_{\fP}\leq Q\}$, then $\fH=\cup_{Q=1}^\infty\fH_Q$. Note that 
    \begin{align*}
        \sup_{\vx\in\Omega}\abs{f_{\vtheta}(\vx)}
        &= \sup_{\vx\in\Omega}\Abs{\sum_{k=1}^\width a_k[\vw_k^\T\mA(\vx)\vw_k\sigma''(\vw_k^\T\vx)+\vb^\T(\vx)\vw_k\sigma'(\vw_k^\T\vx)+c(\vx)\sigma(\vw_k^\T\vx)]}\\
        &\leq \sum_{k=1}^\width\abs{a_k}\norm{\vw_k}_1^3\left[M+\frac{1}{2}M+\frac{1}{6}M\right]\\
        &\leq \frac{5}{3}M\norm{\vtheta}_{\fP}.
    \end{align*}
    Therefore, for functions in $\fH_Q$, since $|f(x)|\leq 1$ by assumption, we have
    \begin{align*}
        0
        \leq \ell(f(\vx),f_{\vtheta} (\vx))
        &\leq \frac{1}{2}(1+\abs{f_{\vtheta}(\vx)})^2\\
        &\leq \frac{1}{2}\left(1+\frac{5}{3}M\norm{\vtheta}_{\fP}\right)^2\\
        &\leq \frac{32}{9}M^2Q^2\leq 4M^2Q^2
    \end{align*}
    for all $\vx\in\Omega$ and all $Q\geq 1$.
    For $\norm{\vtheta}_{\fP}\leq Q$, we note that $\ell(y,\cdot)$ is a Lipschitz function with a Lipschitz constant which is no larger than $\sup_{\vx\in\Omega}\abs{f_{\vtheta}(\vx)}\leq \frac{5}{3}M\norm{\vtheta}_{\fP}+1$.
    Let $S'$ be an arbitrary set of $n$ samples of $\Omega$, then 
    \begin{equation*}
        \Rad_{S'}(\fH_Q)
        \leq (\frac{5}{3}M\norm{\vtheta}_{\fP}+1)\Rad_{S'}(\fF_Q)
        \leq (\frac{5}{3}MQ+1)\Rad_{S'}(\fF_Q).
    \end{equation*}
        
     Let us assume $MQ\geq \frac{3}{5}$ without loss of generality.   By Lemma \ref{lem..RademacherComplexityTwoLayerPathNorm} and Theorem \ref{thm..RademacherComplexityGeneralizationGap}, for any $\delta$ given in Theorem \ref{thm:gen1} and any positive integer $Q$ with probability at least $1-\delta_Q$ over $S$ with $\delta_Q=\frac{6\delta}{\pi^2Q^2}$, we have
    \begin{align*}
        \sup_{\norm{\vtheta}_{\fP}\leq Q}\abs{\RD(\vtheta)-\RS(\vtheta)}
         & \leq (\frac{5}{3}MQ+1)2\Exp_{S'}\Rad_{S'}(\fF_Q)+4M^2Q^2\sqrt{\frac{\log(2/\delta_Q)}{2n}}     \\
         & \leq 27M^2Q^2d^2\sqrt{\frac{2\log(2d)}{n}}+4M^2Q^2\sqrt{\frac{\log(\pi^2Q^2/3\delta)}{2n}}.
    \end{align*}
    
    For any $\vtheta\in\sR^{\width(d+1)}$ given in Theorem \ref{thm:gen1}, choose the integer $Q$ such that $\norm{\vtheta}_{\fP}\leq Q\leq \norm{\vtheta}_{\fP}+1$. Then we have
    \begin{align*}
        \abs{\RD(\vtheta)-\RS(\vtheta)}
         & \leq 27M^2Q^2d^2\sqrt{\frac{2\log(2d)}{n}}+4M^2Q^2\sqrt{\frac{\log(\pi^2Q^2/3\delta)}{2n}}\\
         &\leq 27M^2(\norm{\vtheta}_{\fP}+1)^2d^2\sqrt{\frac{2\log(2d)}{n}}+4M^2(\norm{\vtheta}_{\fP}+1)^2\sqrt{\frac{\log \pi(\norm{\vtheta}_{\fP}+1)}{n}+\frac{\log(1/3\delta)}{2n}}           \\
         & \leq 27M^2(\norm{\vtheta}_{\fP}+1)^2d^2\sqrt{\frac{2\log(2d)}{n}}+4M^2(\norm{\vtheta}_{\fP}+1)^2\left\{\frac{\log[\pi(\norm{\vtheta}_{\fP}+1)]}{\sqrt{n}}+\sqrt{\frac{\log(1/3\delta)}{2n}}\right\}\\
         & \leq \frac{(\norm{\vtheta}_{\fP}+1)^2}{\sqrt{n}}2M^2(14d^2\sqrt{2\log(2d)}+\log[\pi(\norm{\vtheta}_{\fP}+1)]+\sqrt{2\log(1/3\delta)}),
    \end{align*}
    where we have used the facts that $\sqrt{a+b}\leq \sqrt{a}+\sqrt{b}$ for $a,b>0$ and that $\sqrt{a}\leq a$ for $a\geq 1$.
    
    The bound just above holds with probability $1-\delta_Q$ for any pair $(\bm{\theta},Q)$ as long as $\norm{\vtheta}_{\fP}\leq Q$. By the definition $\delta_Q=\frac{6\delta}{\pi^2Q^2}$, we have $\sum_{Q=1}^\infty\delta_Q=\delta$. Therefore, for any $\vtheta\in\sR^{\width(d+1)}$ given in Theorem \ref{thm:gen1}, the above bound holds with probability $1-\delta$, which finishes the proof of Theorem \ref{thm:gen1}.
\end{proof}

Finally, based on the approximation bound in Proposition \ref{prop..ApproximationErrorTwoLayerBarronNorm} and the a posteriori generalization bound in Theorem \ref{thm:gen1}, we show the a priori generalization bound in Theorem \ref{thm:gen2}.


\begin{proof}[Proof of Theorem \ref{thm:gen2}.]
    Note that
    \begin{equation*}
        \RD(\vtheta_{S,\lambda})
        =\RD(\tilde{\vtheta})+[\RD(\vtheta_{S,\lambda})-J_{S,\lambda}(\vtheta_{S,\lambda})]
        +[J_{S,\lambda}(\vtheta_{S,\lambda})-J_{S,\lambda}(\tilde{\vtheta})]
        +[J_{S,\lambda}(\tilde{\vtheta})-\RD(\tilde{\vtheta})].
    \end{equation*}
    By definition,  $J_{S,\lambda}(\vtheta_{S,\lambda})-J_{S,\lambda}(\tilde{\vtheta})\leq 0$.
    By Proposition \ref{prop..ApproximationErrorTwoLayerBarronNorm}, there exists $\tilde{\vtheta}$ such that $\RD(\tilde{\vtheta})\leq \frac{6M^2\norm{f}_{\fB}^2}{\width}$. Therefore,
    \begin{equation}\label{eqn:decom}
    \RD(\vtheta_{S,\lambda})
        \leq \frac{6M^2\norm{f}_{\fB}^2}{\width}+[\RD(\vtheta_{S,\lambda})-J_{S,\lambda}(\vtheta_{S,\lambda})]
        +[J_{S,\lambda}(\tilde{\vtheta})-\RD(\tilde{\vtheta})].
    \end{equation}
    By Theorem \ref{thm:gen1}, we have with probability at least $1-\delta/2$,
    \begin{align}\label{eqn:RDJ}
         \RD(\vtheta_{S,\lambda})-J_{S,\lambda}(\vtheta_{S,\lambda})
         & =\RD(\vtheta_{S,\lambda})-\RS(\vtheta_{S,\lambda})-\frac{\lambda}{\sqrt{n}}\norm{\vtheta_{S,\lambda}}_{\fP}^2\log[\pi(\norm{\vtheta_{S,\lambda}}_{\fP}+1)] \nonumber\\
         & \leq \frac{1}{\sqrt{n}}2M^2(\norm{\vtheta_{S,\lambda}}_{\fP}+1)^2
         \{\log[\pi(\norm{\vtheta_{S,\lambda}}_{\fP}+1)]+14d^2\sqrt{2\log(2d)}+\sqrt{2\log(2/3\delta)}\}\nonumber\\
         &~~-\frac{\lambda}{\sqrt{n}}\norm{\vtheta_{S,\lambda}}_{\fP}^2\log[\pi(\norm{\vtheta_{S,\lambda}}_{\fP}+1)] \nonumber\\
         & \leq \frac{1}{\sqrt{n}}4M^2(\norm{\vtheta_{S,\lambda}}_{\fP}^2+1)\{\log[\pi(\norm{\vtheta_{S,\lambda}}_{\fP}+1)]+14d^2\sqrt{2\log(2d)}+\sqrt{2\log(2/3\delta)}\}\nonumber\\
         &~~-\frac{\lambda}{\sqrt{n}}\norm{\vtheta_{S,\lambda}}_{\fP}^2\log[\pi(\norm{\vtheta_{S,\lambda}}_{\fP}+1)] \nonumber\\
         & \leq \frac{1}{\sqrt{n}}\norm{\vtheta_{S,\lambda}}_{\fP}^2\log[\pi(\norm{\vtheta_{S,\lambda}}_{\fP}+1)] \left\{4M^2[1+14d^2\sqrt{2\log(2d)}+\sqrt{2\log(2/3\delta)}]-\lambda\right\}\nonumber\\
         &~~+\frac{4M^2}{\sqrt{n}}\log[\pi(\norm{\vtheta_{S,\lambda}}_{\fP}+1)]+\frac{1}{\sqrt{n}}4M^2(14d^2\sqrt{2\log(2d)}+\sqrt{2\log(2/3\delta)})\nonumber\\
         & \leq \frac{1}{\sqrt{n}}\norm{\vtheta_{S,\lambda}}_{\fP}^2\log[\pi(\norm{\vtheta_{S,\lambda}}_{\fP}+1)] \left\{4M^2[2+14d^2\sqrt{2\log(2d)}+\sqrt{2\log(2/3\delta)}]-\lambda\right\}\nonumber\\
         &~~+\frac{1}{\sqrt{n}}4M^2\left[\log(2\pi)+14d^2\sqrt{2\log(2d)}+\sqrt{2\log(2/3\delta)}\right]\nonumber\\
         & \leq \frac{1}{\sqrt{n}}4M^2\left[\log(2\pi)+14d^2\sqrt{2\log(2d)}+\sqrt{2\log(2/3\delta)}\right],
    \end{align}
    where we have used the facts that $(a+b)^2\leq 2a^2+2b^2$ for all $a,b\geq 0$ and that $\lambda\geq 4M^2[2+14d^2\sqrt{2\log(2d)}+\sqrt{2\log(2/3\delta)}]$ in the second and last inequalities, respectively.
    By Theorem \ref{thm:gen1} again, with probability at least $1-\delta/2$, we have 
    \begin{align}\label{eqn:JS}
        J_{S,\lambda}(\tilde{\vtheta})-\RD(\tilde{\vtheta})
        &\leq \frac{1}{\sqrt{n}}2M^2(\norm{\tilde{\vtheta}}_{\fP}+1)^2
        \{\log[\pi(\norm{\tilde{\vtheta}}_{\fP}+1)]+14d^2\sqrt{2\log(2d)}+\sqrt{2\log(2/3\delta)}\}\nonumber\\
        &~~+\frac{\lambda}{\sqrt{n}}\norm{\tilde{\vtheta}}_{\fP}^2\log[\pi(\norm{\tilde{\vtheta}}_{\fP}+1)] \nonumber\\
        &\leq \frac{1}{\sqrt{n}}4M^2(\norm{\tilde{\vtheta}}_{\fP}^2+1)
        \{\log[\pi(\norm{\tilde{\vtheta}}_{\fP}+1)]+14d^2\sqrt{2\log(2d)}+\sqrt{2\log(2/3\delta)}\}\nonumber\\
        &~~+\frac{\lambda}{\sqrt{n}}\norm{\tilde{\vtheta}}_{\fP}^2\log[\pi(\norm{\tilde{\vtheta}}_{\fP}+1)].
    \end{align}
    Note that, by Proposition \ref{prop..ApproximationErrorTwoLayerBarronNorm}, we have $\norm{\tilde{\vtheta}}_{\fP}\leq 2\norm{f}_{\fB}$. Hence, the inequality \eqref{eqn:JS} becomes
    \begin{align}
    J_{S,\lambda}(\tilde{\vtheta})-\RD(\tilde{\vtheta})
        &\leq \frac{1}{\sqrt{n}}4M^2(4\norm{f}_{\fB}^2+1)
        \{\log[\pi(2\norm{f}_{\fB}+1)]+14d^2\sqrt{2\log(2d)}+\sqrt{2\log(2/3\delta)}\}\nonumber\\
        &~~+\frac{4\lambda}{\sqrt{n}}\norm{f}_{\fB}^2\log[\pi(2\norm{f}_{\fB}+1)].
    \end{align}
    Adding the estimates in \eqref{eqn:decom}, \eqref{eqn:RDJ}, and \eqref{eqn:JS} together completes the proof.
\end{proof}

\section{Conclusion}
\label{sec:conc}

In this paper, we theoretically analyzed the optimization problem arising in deep learning-based PDE solvers for second-order linear PDEs and two-layer neural networks under the assumption of over-parametrization  (i.e., the network width is sufficiently large). In particular, we show that gradient descent can identify a global minimizer of the least-squares optimization problem for solving second-order linear PDEs. Note that we have fixed the samples in the least-squares optimization, while practical algorithms would randomly sample the PDE domain and its boundaries in every iteration of gradient descent. Hence, there is still a gap between the optimization problem analyzed in this paper and the practical algorithm. This gap can be filled by studying the convergence behavior of stochastic gradient descent, which will be left as future work. 

We have also analyzed the generalization error of deep learning-based PDE solvers for second-order linear PDEs and two-layer neural networks, when the right-hand-side function of the PDE is in a Barron-type space and the least-squares optimization is regularized with a Barron-type norm, without the over-parametrization assumption. The Barron-type space and norm are adaptive the PDE problem and are different from those for regression problems. The global minimizer of the regularized least-squares problem can generalize well with a scaling of order $\frac{1}{\width}+\frac{1}{\sqrt{n}}$, where $\width$ is the number of neurons and $n$ is the number of data samples. Note that whether gradient descent methods can identify a global minimizer of the regularized least-squares problem is still unknown. This is left as interesting future work.

{\bf Acknowledgments.} H. Y. was partially supported by the US National Science Foundation under award DMS-1945029.

\bibliographystyle{plain}
\bibliography{ref}
\end{document}